\documentclass[10pt,a4paper]{article}
\usepackage{amssymb,amsmath,amsopn,amsthm,graphicx,makeidx}
\usepackage[all,2cell]{xy} \UseAllTwocells

\begin{document}

\newtheorem{definition}{Definition}[section]
\newtheorem{definitions}[definition]{Definitions}
\newtheorem{lemma}[definition]{Lemma}
\newtheorem{prop}[definition]{Proposition}
\newtheorem{theorem}[definition]{Theorem}
\newtheorem{cor}[definition]{Corollary}
\newtheorem{cors}[definition]{Corollaries}
\theoremstyle{remark}
\newtheorem{remark}[definition]{Remark}
\theoremstyle{remark}
\newtheorem{remarks}[definition]{Remarks}
\theoremstyle{remark}
\newtheorem{notation}[definition]{Notation}
\theoremstyle{remark}
\newtheorem{example}[definition]{Example}
\theoremstyle{remark}
\newtheorem{examples}[definition]{Examples}
\theoremstyle{remark}
\newtheorem{dgram}[definition]{Diagram}
\theoremstyle{remark}
\newtheorem{fact}[definition]{Fact}
\theoremstyle{remark}
\newtheorem{illust}[definition]{Illustration}
\theoremstyle{remark}
\newtheorem{rmk}[definition]{Remark}
\theoremstyle{definition}
\newtheorem{question}[definition]{Question}
\theoremstyle{definition}
\newtheorem{conj}[definition]{Conjecture}

\newcommand{\stac}[2]{\genfrac{}{}{0pt}{}{#1}{#2}}
\newcommand{\stacc}[3]{\stac{\stac{\stac{}{#1}}{#2}}{\stac{}{#3}}}
\newcommand{\staccc}[4]{\stac{\stac{#1}{#2}}{\stac{#3}{#4}}}
\newcommand{\stacccc}[5]{\stac{\stacc{#1}{#2}{#3}}{\stac{#4}{#5}}}

\renewenvironment{proof}{\noindent {\bf{Proof.}}}{\hspace*{3mm}{$\Box$}{\vspace{9pt}}}

\title{Abelian categories and definable additive categories}
\author{Mike Prest,\\ Alan Turing Building
\\School of Mathematics\\University of Manchester\\
Manchester M13 9PL\\UK\\mprest@manchester.ac.uk}

\maketitle

\tableofcontents

\section{Introduction}\label{intro}

This paper is about three 2-categories which sit an an intersection of algebra, model theory and geometry (the last in the broad sense).

One of these categories, ${\mathbb A}{\mathbb B}{\mathbb E}{\mathbb X}$, has for its objects the skeletally small abelian categories and for its morphisms the exact functors; another, ${\mathbb D}{\mathbb E}{\mathbb F}$, is the category of definable additive categories and interpretation functors; the third is the category, ${\mathbb C}{\mathbb O}{\mathbb H}$, of locally coherent Grothendieck categories and coherent morphisms.  In each case the 2-arrows are just the natural transformations.  The (anti-)equivalences between these were described in \cite{PreRajShv}, which builds on \cite{PreDefAddCat} and \cite{KraFun}, and are recalled below (see also \cite{Mak} and \cite{Hu} for analogous results).  Here these categories and their connections are explored further.

I have tried to include enough explanation of background to make the paper accessible to a variety of readers; for more details one should consult the various references cited.  I will use \cite{PreNBK} as a convenient reference since it gathers together much of what I will need but \cite{HerzCat}, \cite{KraHab}, \cite{PreDefAddCat} also contain much of that.

Throughout this paper, categories are, by default, preadditive, functor means additive functor, $({\mathcal A},{\mathcal B})$ will denote the category of additive functors from the (usually skeletally small) preadditive category ${\mathcal A}$ to the (usually at least additive) category ${\mathcal B}$, ${\bf Ab}$ will denote the category of abelian groups, ${\rm Mod}\mbox{-}{\mathcal A}$ will be an alternative notation for $({\mathcal A}^{\rm op}, {\bf Ab})$ (where $^{\rm op}$ denotes the opposite of a category) - it is the category of right ${\mathcal A}$-modules - and ${\mathcal A}\mbox{-}{\rm Mod} =({\mathcal A}, {\bf Ab})$ will denote the category of left ${\mathcal A}$-modules.  The full subcategory of finitely presented modules is denoted by ${\rm mod}\mbox{-}{\mathcal A}$.  We write ${\mathbb P}{\mathbb R}{\mathbb E}{\mathbb A}{\mathbb D}{\mathbb D}$ for the 2-category of preadditive categories (additive functors and natural transformations).  We scarcely distinguish between a skeletally small category and a small version of it (i.e.~a category to which it is equivalent but which has just a set of objects).

\vspace{4pt}

Now we show how the three 2-categories are related, then give a quick summary of what is in each section.

\begin{theorem}\label{3cats} \cite[2.3 and comments following]{PreRajShv} There is a diagram of equivalences and anti-equivalences between $ {\mathbb A}{\mathbb B}{\mathbb E}{\mathbb X}$, $ {\mathbb C}{\mathbb O}{\mathbb H} $ and $ {\mathbb D}{\mathbb E}{\mathbb F} $ as follows.

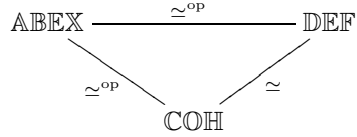
\begin{center} $\xymatrix{{\mathbb A}{\mathbb B}{\mathbb E}{\mathbb X} \ar@{-}[rr]^{\simeq^{\rm op}} \ar@{-}[rd]_{\simeq^{\rm op}} & & {\mathbb D}{\mathbb E}{\mathbb F} \\ & {\mathbb C}{\mathbb O}{\mathbb H} \ar@{-}[ur]_{\simeq}}$
\end{center}

\noindent Explicitly:

\begin{center} $
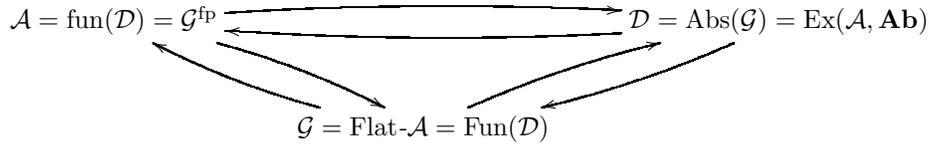
\xymatrix{{\mathcal A} ={\rm fun}({\mathcal D}) ={\mathcal G}^{\rm fp}  \ar@/^/[rr] \ar@/^/[dr] & & {\mathcal D} = {\rm Abs}({\mathcal G}) = {\rm Ex}({\mathcal A}, {\bf Ab})   \ar@/^/[ll] \ar@/^/[dl] \\ & {\mathcal G} = {\rm Flat}\mbox{-}{\mathcal A} = {\rm Fun}({\mathcal D}) \ar@/^/[ul] \ar@/^/[ur]}$
\end{center}

\end{theorem}

We will need the details of these (anti-)equivalences, so here they are.

\vspace{4pt}

\noindent From ${\mathbb A}{\mathbb B}{\mathbb E}{\mathbb X}$ to ${\mathbb D}{\mathbb E}{\mathbb F}$: to a skeletally small abelian category ${\mathcal A}$ we associate the definable category ${\rm Ex}({\mathcal A},{\bf Ab})$ - the full subcategory of ${\mathcal A}\mbox{-}{\rm Mod}$ on those functors which are exact; to an exact functor $F:{\mathcal A} \rightarrow {\mathcal B}$, we associate the functor $F^\ast:{\rm Ex}({\mathcal B},{\bf Ab}) \rightarrow {\rm Ex}({\mathcal A},{\bf Ab})$ which is just precomposition with $F$.

\vspace{4pt}

\noindent From ${\mathbb D}{\mathbb E}{\mathbb F}$ to ${\mathbb A}{\mathbb B}{\mathbb E}{\mathbb X}$: to a definable category ${\mathcal D}$ we associate the category, ${\rm fun}({\mathcal D}) =({\mathcal D}, {\bf Ab})^{\rightarrow \prod}$, of functors from ${\mathcal D}$ which commute with direct limits and direct products (we write ${\rm fun}\mbox{-}{\mathcal R}$ in the case that ${\mathcal D}={\rm Mod}\mbox{-}{\mathcal R}$); to an interpretation functor, that is, a functor $I:{\mathcal C} \rightarrow {\mathcal D}$ which commutes with direct products and direct limits, we associate the functor $I_0:{\rm fun}({\mathcal D}) \rightarrow {\rm fun}({\mathcal C})$ which is precomposition with $I$.

\vspace{4pt}

\noindent Between ${\mathbb A}{\mathbb B}{\mathbb E}{\mathbb X}$ and ${\mathbb C}{\mathbb O}{\mathbb H}$ (on objects): to a locally coherent Grothendieck category ${\mathcal G}$ we assign its full subcategory, ${\mathcal G}^{\rm fp}$, of finitely presented objects; in the other direction, to a skeletally small abelian category ${\mathcal A}$ we assign the category ${\rm Lex}({\mathcal A}^{\rm op},{\bf Ab})$ of left exact functors on ${\mathcal A}^{\rm op}$, thus right exact functors on ${\mathcal A}$, so this includes the representable functors $(-,A)$ for $A\in {\mathcal A}$.  This is a locally coherent Grothendieck category and the image of ${\mathcal A}$ under the just-mentioned Yoneda embedding $A\mapsto (-,A)$ is equivalent to the full subcategory of finitely presented objects (see \ref{abtofromloccoh}, also for the identifications ${\rm Lex}({\mathcal A}^{\rm op}, {\bf Ab}) \simeq {\rm Flat}\mbox{-}{\mathcal A} \simeq {\rm Ind}({\mathcal A})$).

\vspace{4pt}

\noindent Between ${\mathbb A}{\mathbb B}{\mathbb E}{\mathbb X}$ and ${\mathbb C}{\mathbb O}{\mathbb H}$ (on morphisms): from a morphism $f\in {\rm Ex}({\mathcal A}, {\mathcal B})$ we define the coherent morphism (see Section \ref{seccoh}) $(f^\ast,f_\ast):{\mathcal H}={\rm Ind}({\mathcal B}) \rightarrow {\rm Ind}({\mathcal A})={\mathcal G}$ which has $f_\ast:{\mathcal H}={\rm Lex}({\mathcal B}^{\rm op}, {\bf Ab}) \rightarrow {\rm Lex}({\mathcal A}^{\rm op}, {\bf Ab})={\mathcal G}$ just precomposition with $f^{\rm op}$ and has $f^\ast ={\rm Ind}(f)$.  In the other direction we take a coherent morphism $(f^\ast,f_\ast)$ to the restriction of the left adjoint, $f^\ast$, to the finitely presented objects of ${\mathcal G}$.

\vspace{4pt}  The as-yet-unexplained notation ${\rm Abs}({\mathcal G})$ refers to the full subcategory of {\bf absolutely pure} (or {\bf fp-injective}) objects of ${\mathcal G}$ - those objects $G$ such that ${\rm Ext}^1({\mathcal G}^{\rm fp},G)=0$.

\vspace{4pt}

The next result, which is not difficult to show (or see \cite{PreAxtFlat}), is one instance of this picture.  By ${\mathcal A}({\mathcal R})$ we denote the smallest abelian (not necessarily full) subcategory of ${\rm Mod}\mbox{-}{\mathcal R}$ which contains ${\rm mod}\mbox{-}{\mathcal R}$ (see \cite[\S 6]{PreRajShv}, also Section \ref{secabsch} below).  By $\langle {\mathcal X}\rangle$ we denote the smallest definable subcategory of ${\rm Mod}\mbox{-}{\mathcal R}$ containing ${\mathcal X}$.

\begin{prop}\label{cohexgendual} If $ {\mathcal R} $ is any skeletally small preadditive category then $${\rm Ex}({\mathcal A}({\mathcal R})^{\rm op},{\bf Ab})\simeq \langle  {\rm Abs}\mbox{-}{\mathcal R}\rangle  .$$ If $ {\mathcal R} $ is right coherent, so $ {\rm Abs}\mbox{-}{\mathcal R}$ is a definable subcategory of $ {\rm Mod}\mbox{-}{\mathcal R}$, then $${\rm Ex}(({\rm mod}\mbox{-}{\mathcal R})^{\rm op},{\bf Ab})\simeq {\rm Abs}\mbox{-}{\mathcal R}.$$
\end{prop}

Note the duality which applies to the whole picture described above. It is obvious for $ {\mathbb A}{\mathbb B}{\mathbb E}{\mathbb X}$, on which it is the 2-category equivalence which takes each abelian category to its opposite. It follows that there is a corresponding self-equivalence on each of the other two categories (which will be described in the relevant section).  In the context of the model theory of definable subcategories of module categories this duality was found first for pp formulas, and termed elementary duality (\cite{PreDual}) then extended to the category of pp-pairs and the Ziegler spectrum in \cite{HerzDual}.  In an algebraic form it is in \cite{AusDual} and \cite{GrJeDim}.

\vspace{4pt}

For instance, the dual to the result above is the following.

\begin{prop}\label{cohexgen} If $ {\mathcal R} $ is any skeletally small preadditive category then $${\rm Ex}({\mathcal A}({\mathcal R}),{\bf Ab})\simeq \langle {\mathcal R}\mbox{-}{\rm Flat}\rangle .$$ If $ {\mathcal R} $ is right coherent, so $  {\mathcal A}({\mathcal R})={\rm mod}\mbox{-}{\mathcal R} $ and $ {\mathcal R}\mbox{-}{\rm Flat} $ is a definable subcategory of $ {\mathcal R}\mbox{-}{\rm Mod}$, then $${\rm Ex}({\rm mod}\mbox{-}{\mathcal R},{\bf Ab})\simeq {\mathcal R}\mbox{-}{\rm Flat}.$$
\end{prop}

In Section \ref{secabex} we identify the finitely presented objects of $ {\mathbb A}{\mathbb B}{\mathbb E}{\mathbb X}$ as the finite type localisations of free abelian categories of finitely presented rings and we show that every small abelian category is a direct limit - ``directed colimit" in the more category-theoretic terminology - of such categories.  Since $ {\mathbb A}{\mathbb B}{\mathbb E}{\mathbb X}$ also has directed colimits in a suitable 2-category sense we could therefore say that $ {\mathbb A}{\mathbb B}{\mathbb E}{\mathbb X}$ is finitely accessible (in some 2-category sense).  We show that $ {\mathbb A}{\mathbb B}{\mathbb E}{\mathbb X}$ has pullbacks and also characterise the monomorphisms (and say a little about the epimorphisms) of this category.

The main result of Section \ref{secdefcats} is that the structure of ${\mathbb D}{\mathbb E}{\mathbb F}$ - arrows and 2-arrows as well as the objects - is essentially determined by the full subcategories of pure-injective objects.  We also show that if ${\mathcal A}$ is skeletally small abelian then ${\rm Ex}({\mathcal A}, {\mathcal D})$ is definable for any  definable Grothendieck (so, \ref{defgroth}, locally finitely presented) category ${\mathcal D}$ (not just when ${\mathcal D}= {\bf Ab}$).

Section \ref{seccoh} is devoted to developing an additive version of things (coherent morphisms, classifying toposes, points) that are familiar in the context of toposes.  The parallel is well-known but we develop it further here as part of the larger additive picture.

\section{The category of small abelian categories and exact functors}\label{secabex}

The category $ {\mathbb A}{\mathbb B}{\mathbb E}{\mathbb X} $, of skeletally small abelian categories and exact functors, belongs to algebra but it has a model-theoretic meaning (its objects are the categories of pp-sorts and pp-definable functions for the corresponding definable categories, see \cite[Chpt.~22]{PreDefAddCat} or \cite[Part III]{PreNBK}, also \cite{PreAxtFlat}).  It can also be seen as generalising the category of rings through the free abelian category construction but also through a possibly more geometric construction (see Section \ref{secrngtoab}).

\vspace{4pt}

Recall that an {\bf exact} functor $  F:{\mathcal A}\rightarrow {\mathcal B} $ between abelian categories is a functor such that if $ 0\rightarrow A'\rightarrow A\rightarrow A''\rightarrow 0 $ is an exact sequence in $ {\mathcal A} $ then $ 0\rightarrow FA'\rightarrow FA\rightarrow FA''\rightarrow 0 $ is an exact sequence in $ {\mathcal B}$.  By the {\bf kernel} of such a functor we mean the full subcategory ${\rm ker}(F)$ on the objects $\{ A\in {\mathcal A}: FA=0\}$ which are sent to $0$ by $F$.  This is a {\bf Serre subcategory} of ${\mathcal A}$, meaning a full subcategory ${\mathcal S}$ of ${\mathcal A}$ which is closed under subobjects, quotient objects and extensions: otherwise said, if $ 0\rightarrow A'\rightarrow A\rightarrow A''\rightarrow 0 $ is exact then $A\in {\mathcal S}$ iff $A', A''\in {\mathcal S}$.  Conversely every Serre subcategory ${\mathcal S}$ is the kernel of an exact functor from ${\mathcal A}$.  A {\bf localisation} of ${\mathcal A}$ is an exact functor $F:{\mathcal A} \rightarrow {\mathcal B}$ which is such that the image of $F$ is full and includes an isomorphic copy of every object in ${\mathcal B}$; that is, up to equivalence ${\mathcal B}$ is the image of $F$.  We define the {\bf quotient category} of ${\mathcal A}$ at the Serre subcategory ${\mathcal S}$ to have the same objects as ${\mathcal A}$ but define the morphisms from $A$ to $B$ in ${\mathcal A}/{\mathcal S}$ to be equivalence classes of morphisms from subobjects $A'$ of $A$, with $A/A'\in {\mathcal S}$, to factor objects $B/B'$ of $B$, with $B'\in {\mathcal S}$, under a natural (eventual agreement) equivalence relation.  See, for instance, \cite[\S IX.1]{Ste}, for details but the idea is simply that one forces the objects in ${\mathcal S}$ to become zero.  If $F:{\mathcal A} \rightarrow {\mathcal B}$ is a localisation then ${\mathcal B}$ is equivalent to ${\mathcal A}/{\rm ker}(F)$ and $F$ has a right adjoint which is a full, though not in general exact, embedding $i$ of ${\mathcal B}$ in ${\mathcal A}$.  Thus the image of $i$ is a reflective subcategory of ${\mathcal A}$.  We write $\langle {\mathcal X} \rangle$ for the Serre subcategory generated by a collection ${\mathcal X}$ of objects of ${\mathcal A}$.  We also write ${\rm Ser}({\mathcal A})$ for the set of Serre subcategories of ${\mathcal A}$.

We will need the following theorem, the first paragraph of which is \cite[4.1]{FreydLJ} (for fuller references see \cite[\S10.2.7]{PreNBK} or \cite[\S4]{PreDefAddCat}).

\begin{theorem}\label{freeabthm} Let ${\mathcal R}$ be a skeletally small preadditive category.  Then there is an additive functor $ i $ from ${\mathcal R}$ to a skeletally small abelian category $ {\rm Ab}({\mathcal R}) $ such that if $ \alpha :{\mathcal R}\rightarrow {\mathcal B} $ is any additive functor to an abelian category ${\mathcal B}$ then there is a factorisation through $ i $ {\it via} a unique-to-natural-equivalence exact functor $ {\rm Ab}({\mathcal R})\rightarrow {\mathcal B}$.

$$\xymatrix{{\mathcal R} \ar[r]^i \ar[rd]_{\alpha} & {\rm Ab}({\mathcal R}) \ar[d]^{\alpha '} \\ & {\mathcal B}}$$

The category $ {\rm Ab}({\mathcal R}) $ may be identified with $ ({\mathcal R}\mbox{-}{\rm mod}, {\bf Ab})^{\rm fp} \simeq \big(({\rm mod}\mbox{-}{\mathcal R},{\bf Ab})^{\rm fp}\big)^{\rm op}$ and the embedding $ i $ takes an object $ P $ of $ {\mathcal R} $ to $((P,-),-)$ and has the then obvious action on morphisms.
\end{theorem}

\begin{cor}\label{freeabdual} If ${\mathcal R}$ is a skeletally small preadditive category then $ {\rm Ab}({\mathcal R}^{\rm op}) \simeq {\rm Ab}({\mathcal R})^{\rm op}$.
\end{cor}

The category ${\rm Ab}({\mathcal R})$, more precisely the functor ${\mathcal R} \rightarrow {\rm Ab}({\mathcal R})$, given by \ref{freeabthm} is the {\bf free abelian category} on ${\mathcal R}$.  In the case that we start with a ring $R$, that is a preadditive category with one object $\ast_R$ which has endomorphism ring $R$, the image of that object in ${\rm Ab}(R)$ is the representable functor $((\ast_R,-),-)$ on the representable functor $(\ast_R,-) \in R\mbox{-}{\rm mod}$ (this latter being the projective left module $_RR$).

Note that, taking ${\mathcal B} ={\bf Ab}$ and allowing $\alpha$ to roam over all (covariant) functors, i.e.~left ${\mathcal R}$-modules, we obtain that $ {\mathcal R}\mbox{-}{\rm Mod}$ is equivalent to $ {\rm Ex}({\rm Ab}({\mathcal R}),{\bf Ab}) $.  Replacing ${\mathcal R}$ by ${\mathcal R}^{\rm op}$ to get the contravariant form, and noting \ref{freeabdual}, we have ${\rm Mod}\mbox{-}{\mathcal R} \simeq {\rm Ex}({\rm Ab}({\mathcal R})^{\rm op}, {\bf Ab})$.

If a ring $R$ is von Neumann regular then ${\rm Ab}(R)\simeq {\rm mod}\mbox{-}R$, indeed these are exactly the rings for which this is true (see, e.g., \cite[10.2.38]{PreNBK} (where the statement is missing an $^{\rm op}$)).

\subsection{Categorical properties of $ {\mathbb A}{\mathbb B}{\mathbb E}{\mathbb X}$}\label{seccatabex}

If $ {\mathcal A}, {\mathcal B}\in {\mathbb A}{\mathbb B}{\mathbb E}{\mathbb X}$ then $ {\rm Ex}({\mathcal A},{\mathcal B}) $ is a skeletally small category with objects the exact functors from $ {\mathcal A} $ to $ {\mathcal B} $ and with morphisms the natural transformations between these.  It is a preadditive category:  for $ F,G\in {\rm Ex}({\mathcal A},{\mathcal B}) $ the set $ (F,G)={\rm Nat}(F,G) $ of natural transformations from $ F $ to $ G $ is an abelian group with $ 0\in {\rm Nat}(F,G) $ being given by $ 0_A=0:FA\rightarrow GA $ for each object $ A\in {\mathcal A} $ and, if $ \tau ,\mu \in {\rm Nat}(F,G) $ then $ \tau -\mu  $ is defined at $ A\in {\mathcal A} $ by $ (\tau -\mu )_A=\tau _A-\mu _A:FA\rightarrow GA$. Indeed, $ {\rm Ex}({\mathcal A},{\mathcal B}) $ is additive: for a zero object, choose a zero object $ 0_{{\mathcal B}}\in {\mathcal B} $ and define the functor $ 0:{\mathcal A}\rightarrow {\mathcal B} $ by $ A\mapsto 0_{{\mathcal B}} $ for all $A\in {\mathcal A}$. And if $ F,G\in {\rm Ex}({\mathcal A},{\mathcal B}) $ then we may, since $ {\mathcal B} $ has finite direct sums, define $ F\oplus G $ by taking $ A\in {\mathcal A} $ to $ FA\oplus GA $ and $ f:A\rightarrow A' $ to $ (Ff,Gf):FA\oplus GA\rightarrow FA'\oplus GA'$.  Thus $ {\mathbb A}{\mathbb B}{\mathbb E}{\mathbb X}$ may be seen as a category enriched in additive categories.

\begin{example}\label{notab} $({\mathcal A},{\mathcal B})$ need not be abelian.

Let $ R $ be right coherent, so $ {\rm mod}\mbox{-}R\in {\mathbb A}{\mathbb B}{\mathbb E}{\mathbb X} $ and $ {\rm Ex}({\rm mod}\mbox{-}R,{\bf Ab})\simeq R\mbox{-}{\rm Flat}$, the category of flat left $R$-modules (see \ref{cohexgen}). In particular $ {\rm Ex}({\rm mod}\mbox{-}{\mathbb Z},{\bf Ab})\simeq {\mathbb Z}\mbox{-}{\rm Flat} $ ($\ast $) and, we claim, $ {\rm Ex}({\rm mod}\mbox{-}{\mathbb Z},{\rm mod}\mbox{-}{\mathbb Z})\simeq {\mathbb Z}\mbox{-}{\rm Flat}\cap {\mathbb Z}\mbox{-}{\rm mod}={\mathbb Z}\mbox{-}{\rm proj} $, the category of finitely generated projective ${\mathbb Z}$-modules, which is not abelian. To see the claim we just use that in the equivalence ($\ast $) a flat module $M$ on the right-hand side of the equivalence acts on ${\rm mod}\mbox{-}{\mathbb Z}$ as the (exact) functor $-\otimes M$; it is clear that if such a functor outputs only finitely presented values then $M$ is finitely presented, and conversely.
\end{example}

Bearing in mind that $ {\mathbb A}{\mathbb B}{\mathbb E}{\mathbb X} $ is a 2-category, so equality is generally replaced by natural equivalence, we will say that an exact functor $ F:{\mathcal A}\rightarrow {\mathcal B} $ in $ {\mathbb A}{\mathbb B}{\mathbb E}{\mathbb X} $ is a {\bf monomorphism} if for all $ G,H:{\mathcal A}'\rightarrow {\mathcal A} $ if $ \tau :FG\rightarrow FH $ is a natural isomorphism then there is a natural isomorphism $ \eta:G\rightarrow H $ such that $ \tau =F\eta$.

\begin{lemma}\label{morfromZ} For every $ {\mathcal A}\in {\mathbb A}{\mathbb B}{\mathbb E}{\mathbb X} $ and $ A\in {\mathcal A} $ there is a morphism $ {\rm Ab}({\mathbb Z})\rightarrow {\mathcal A} $ in $ {\mathbb A}{\mathbb B}{\mathbb E}{\mathbb X}$ such that the single object, $((\ast_{\mathbb Z},-),-)$, of the image of $ {\mathbb Z} $ in ${\rm Ab}({\mathbb Z})$ (\ref{freeabthm}) is taken to $ A $ (and hence the ring, $ {\mathbb Z} $, of endomorphisms of $ \ast_{\mathbb Z}  $ is taken to its natural image in $ {\rm End}(A)$).
\end{lemma}
\begin{proof} Define the functor by taking $ \ast =\ast_{\mathbb Z} $ to $ A $ (and extending $1_{\mathbb Z} ={\rm id}_\ast \mapsto {\rm id}_A$ to a ring homomorphism) and then we have an extension to an exact functor (unique to natural equivalence) $ {\rm Ab}({\mathbb Z})\rightarrow {\mathcal A} $ by \ref{freeabthm}.
\end{proof}

\begin{lemma}\label{abexmono} If $ F:{\mathcal A}\rightarrow {\mathcal B} $ is a monomorphism in $ {\mathbb A}{\mathbb B}{\mathbb E}{\mathbb X}$ then $ F $ is full on isomorphisms in the strong sense that if $g:FA_1 \rightarrow FA_2$ is an isomorphism then there is an isomorphism $f:A_1\rightarrow A_2$ such that $Ff=g$.  This second condition on $F$ is equivalent to $F$ being faithful and every isomorphism $g:FA_1 \rightarrow FA_2$ being the image of some morphism $A_1\rightarrow A_2$.
\end{lemma}
\begin{proof} Suppose that $ A_1,A_2\in {\mathcal A} $ and that $ g:FA_1\rightarrow FA_2 $ is an isomorphism in $ {\mathcal B}$.  Let $ G_i:{\rm Ab}({\mathbb Z})\rightarrow {\mathcal A} $ be as in the lemma above, taking $ \ast  =\ast_{\mathbb Z}$ to $ A_i$; then, by \ref{freeabthm}, we deduce that there is a natural isomorphism $ \tau :FG_1\rightarrow FG_2 $ essentially determined by $ \tau _\ast =g $ so, by assumption, there is a natural isomorphism $ \eta:G_1\rightarrow G_2$, in particular an isomorphism $ \eta_\ast :A_1\rightarrow A_2 $, such that $F\eta_\ast =g$, as required.

For the equivalent condition, recall that a functor $F$ is {\bf faithful} if whenever $f,f':A_1\rightarrow A_2$ are in its domain and $Ff=Ff'$ then $f=f'$.  In the case of an additive functor, one replaces the pair $f,f'$ by $f-f',0$ and the condition becomes $Ff=0$ implies $f=0$.  In the case that $F$ is an exact functor between abelian categories then, since $F{\rm (co)ker}(f) ={\rm (co)ker}(Ff)$, one deduces that $F$ is faithful iff the object kernel is zero, that is iff $FA=0$ implies $A=0$.  So if $F$ {\bf reflects isomorphism}, meaning that $FA_1\simeq FA_2$ implies $A_1\simeq A_2$ then $F$ is faithful.  For the converse, if $g:FA_1\rightarrow FA_2$ is an isomorphism then, by assumption, there is a morphism $f:A_1\rightarrow A_2$ with $g=Ff$.  If either ${\rm ker}(f)$ or ${\rm coker}(f)$ were non-zero then, by exactness and faithfulness of $F$, $Ff$ would have a non-zero kernel or cokernel, contradiction.  So $f$ is an isomorphism.
\end{proof}

\begin{prop}\label{fulliso} The following are equivalent for a morphism $ F:{\mathcal A}\rightarrow {\mathcal B} $ in $ {\mathbb A}{\mathbb B}{\mathbb E}{\mathbb X} $:

\noindent (i) $F$ is monic in $ {\mathbb A}{\mathbb B}{\mathbb E}{\mathbb X} $;

\noindent (ii) $F$ is full on isomorphisms in the strong sense of \ref{abexmono}

\noindent (iii) $F$ is faithful and full on isomorphisms.
\end{prop}
\begin{proof} (i)$\Leftrightarrow$(ii) One direction is the lemma above. For the other suppose that $ F $ is full on isomorphisms in the strong sense and that we have $ G,H:{\mathcal A}'\rightarrow {\mathcal A} $ and a natural equivalence $ \tau :FG\rightarrow FH$. So for each $ A'\in {\mathcal A}' $ we have an isomorphism $ \tau _{A'}:FGA'\rightarrow FHA'$; by assumption there is an isomorphism which we will denote $ \eta_{A'}:GA'\rightarrow HA' $ such that $ F\eta_{A'}=\tau _{A'}$. We have to show that the $ \eta_{A'} $ fit together to form a natural transformation from $ G $ to $ H$.

Given $ f:A'\rightarrow A'' $ in $  {\mathcal A} $ we consider the diagram

$\xymatrix{GA' \ar[r]^{\eta_{A'}} \ar[d]_{Gf} & HA' \ar[d]^{Hf} \\ GA'' \ar[r]_{\eta_{A''}} & HA''}$

\noindent and the difference $ Hf.\eta_{A'}-\eta_{A''}.Gf$; apply $ F $ to obtain the diagram

$\xymatrix{FGA' \ar[r]^{F\eta_{A'}=\tau_{A'}} \ar[d]_{GGf} & FHA' \ar[d]^{FHf} \\ FGA'' \ar[r]_{F\eta_{A''}=\tau_{A''}} & FHA''}$

\noindent which commutes, that is $ F(Hf.\eta_{A'}-\eta_{A''}.Gf)=0$, so it will suffice to show that the kernel of $ F $ in the sense of arrows is zero.  But we have this from (the proof of) \ref{abexmono}, which also gives us the equivalence of (ii) and (iii).
\end{proof}

So a monomorphism in $ {\mathbb A}{\mathbb B}{\mathbb E}{\mathbb X} $ is, in particular, an embedding of an abelian (i.e.~exact) subcategory and any such embedding which is full (e.g.~the embedding of the category of finite abelian groups in the category of finitely generated ones) is a monomorphism.

\begin{example}\label{mononotker} A monomorphism in $ {\mathbb A}{\mathbb B}{\mathbb E}{\mathbb X} $ need not be a kernel.

Let $ R $ be a tame hereditary finite-dimensional algebra, set $ {\mathcal B}={\rm mod}\mbox{-}R $ and let $ {\mathcal A} $ be the full, abelian, subcategory of regular modules. Then $ {\mathcal A}\rightarrow {\mathcal B} $ is a morphism in $ {\mathbb A}{\mathbb B}{\mathbb E}{\mathbb X} $ and is a full embedding so certainly is a monomorphism in $ {\mathbb A}{\mathbb B}{\mathbb E}{\mathbb X}$. If it were the kernel of an exact functor then ${\mathcal A}$ would be a Serre subcategory - not so since $ R_R $ embeds in a regular module and so the Serre subcategory generated by $ {\mathcal A} $ is all of $ {\mathcal B}$.
\end{example}

\begin{example}\label{abexmononotfull} A monomorphism in $ {\mathbb A}{\mathbb B}{\mathbb E}{\mathbb X} $ need not be full.  Consider the non-full functor $F:{\rm mod}\mbox{-}k\rightarrow {\rm mod}\mbox{-}kA_2$, where $A_2$ is the quiver $\bullet \rightarrow \bullet$, which takes $k_k$ to the $kA_2$-module ${kA_2}$.  If  $P, Q$ are in the image of $F$, that is, are finitely generated free $kA_2$-modules, then it is easily computed that any isomorphism $P\rightarrow Q$ is in the image of $F$.  So, by \ref{fulliso}, $F$ is a monomorphism in $ {\mathbb A}{\mathbb B}{\mathbb E}{\mathbb X} $ but clearly $F$ is not full.
\end{example}

\begin{example}\label{fullisonotmono} Consider the functor $ {\rm mod}\mbox{-}kA_2\rightarrow {\rm mod}\mbox{-}k $ which takes a representation $ V_1\xrightarrow{\alpha _V}V_2 $ of the quiver $A_2$ to $ V_1\oplus V_2 $ and which has the obvious action on morphisms (i.e.~takes $ (f_1,f_2) $ to $ f_1\oplus f_2$). This functor is clearly exact and faithful but is not full on isomorphisms hence is not a monomorphism in $ {\mathbb A}{\mathbb B}{\mathbb E}{\mathbb X}$.
\end{example}

Say that $ F:{\mathcal A}\rightarrow {\mathcal B} $ in $ {\mathbb A}{\mathbb B}{\mathbb E}{\mathbb X}$ is an {\bf epimorphism} if whenever $ G,H:{\mathcal B}\rightarrow {\mathcal C}$ are such that there is a natural equivalence $ \tau :GF\rightarrow HF $ then there is a natural equivalence $ \mu :G\rightarrow H $ such that $ \mu F=\tau $.

\begin{example}\label{defnotepi} A morphism $ F:{\mathcal A}\rightarrow {\mathcal B} $ in $ {\mathbb A}{\mathbb B}{\mathbb E}{\mathbb X} $ with $ \langle F{\mathcal A}\rangle ={\mathcal B} $ need not be an epimorphism.

Let $ F:{\rm mod}\mbox{-}(k\times k) \rightarrow {\rm mod}\mbox{-}kA_2 $ be the functor which takes $ (V_1,V_2) $ to the representation $ V_1\xrightarrow{0}V_2 $ of $ A_2$ and consider the functors $ {\rm Id},G:{\rm mod}\mbox{-}k[A_2]\rightarrow {\rm mod}\mbox{-}kA_2 $ where $ G $ takes a representation $ W_1\xrightarrow{\alpha }W_2 $ to $ W_1\xrightarrow{0}W_2 $.  Both $F$ and $G$ have the obvious definitions on morphisms. It is easily checked that these functors are exact and clearly $ GF={\rm Id}.F $ but there is no natural equivalence $ {\rm Id}\rightarrow G$, so $ F $ is not an epimorphism.

On the other hand, each of the two simple representations $ k\rightarrow 0 $ and $ 0\rightarrow k $ of $A_2$ is in the image of $ F $ so, since every object of ${\rm mod}\mbox{-}kA_2$ has finite length, the Serre subcategory of $ {\rm mod}\mbox{-}k[A_2] $ generated by the image of $ F $ is all of the category.
\end{example}

We do not have a characterisation of epimorphisms in $ {\mathbb A}{\mathbb B}{\mathbb E}{\mathbb X} $, but note the following.

\begin{prop}\label{episindef} (1) If $ {\mathcal A}\in {\mathbb A}{\mathbb B}{\mathbb E}{\mathbb X}$ and $ {\mathcal S}\in {\rm Ser}({\mathcal A})$ then the canonical localisation functor $ \pi :{\mathcal A}\rightarrow {\mathcal A}/{\mathcal S} $ is an epimorphism in $  {\mathbb A}{\mathbb B}{\mathbb E}{\mathbb X}$.

\noindent (2) If $ F:{\mathcal A}\rightarrow {\mathcal B} $ is an epimorphism in $ {\mathbb A}{\mathbb B}{\mathbb E}{\mathbb X}$ then the Serre subcategory, $ \langle F{\mathcal A}\rangle$, generated by $F{\mathcal A}$ is ${\mathcal B}$.
\end{prop}
\begin{proof} (1) If $ G,H:{\mathcal A}/{\mathcal S}\rightarrow {\mathcal B} $ are such that there is a natural equivalence $ \tau :G\pi \rightarrow H\pi  $ then for $ A\in {\mathcal A} $ we have an isomorphism $ \tau _{{\mathcal A}}:G\pi A\rightarrow H\pi A$. Since we may regard $ {\mathcal A}/{\mathcal S} $ as having objects those of $ {\mathcal A} $ (but different morphisms) we can define $ \eta:G\rightarrow H $ to have component $ \eta_{\pi A}=\tau _A $ at $ \pi A$. It must be checked that this is well-defined and that these components cohere to form a natural transformation. Both come down to showing that if $ f:\pi A\rightarrow \pi A' $ then $ Hf.\tau _A=\tau _{A'}.Gf$. But $ f\simeq \pi g $ for some $g:A_1\rightarrow A_1'$ where there is a monomorphism $A_1\rightarrow A$ with cokernel in ${\mathcal S}$ and an epimorphism $A' \rightarrow A_1'$ with kernel in ${\mathcal S}$, in the sense that $f$ is the composition shown (using the indicated inverses of those morphisms, both of which become invertible in ${\mathcal A}/{\mathcal S}$).

$\xymatrix{A_1 \ar[r] \ar[d]_g & A \\ A_1' & A' \ar[l]}$  $\xymatrix{\pi A_1 \ar@/_/[r]_{\simeq} \ar[d]_{\pi g} & \pi A \ar@{.>}@/_/[l]\ar@{.>}[d]^f \\ \pi A_1' \ar@{.>}@/_/[r]_{\simeq} & \pi A' \ar@/_/[l]}$

We then apply $G$ and $H$ to obtain the diagram shown next

$\xymatrix{G\pi A \ar[rrr]^{Gf} \ar[ddd]_{\tau_A} & & & G\pi A' \ar[ddd]^{\tau_{A'}} \ar@{<->}[dl] \\ & G\pi A_1 \ar[r]^{G\pi g} \ar@{<->}[ul] \ar[d]_{\tau_{A_1}} & G \pi A_1' \ar[d]^{\tau_{A_1'}} \\ & H\pi A_1 \ar[r]^{H\pi g} \ar@{<->}[dl] & H \pi A_1' \\ H\pi A \ar[rrr]_{Hf} & & & H\pi A' \ar@{<->}[ul] }$

The smaller quadrilaterals commute and so, therefore, does the outer one, as required.

(2) Consider $\xymatrix{{\mathcal A} \ar[r]^F & {\mathcal B} \ar@/^/[r]^\pi \ar@/_/[r]_0 & {\mathcal B}/\langle F{\mathcal A} \rangle}$.  Since $\pi F =0= 0.F$ we must have $\pi \simeq 0$ hence $\langle F{\mathcal A} \rangle = {\mathcal B}$.
\end{proof}

\subsection{Pullbacks in $ {\mathbb A}{\mathbb B}{\mathbb E}{\mathbb X}$}\label{secabexpbk}

We show that $ {\mathbb A}{\mathbb B}{\mathbb E}{\mathbb X}$ has pullbacks, hence a notion of ``base change" (see Section \ref{secrngtoab}).

Given a diagram

$\xymatrix{  & {\mathcal A} \ar[d]^{F} \\ {\mathcal B} \ar[r]_{G} & {\mathcal C} }$

\noindent in $ {\mathbb A}{\mathbb B}{\mathbb E}{\mathbb X}$, we construct the following category $ {\mathcal P}$. The objects of $ {\mathcal P} $ are triples $ (A,B,\gamma ) $ where $ A $ is an object of $ {\mathcal A}$, $ B $ an object of $ {\mathcal B} $ and $ \gamma :FA\rightarrow GB $ is an isomorphism in $ {\mathcal C}$. A morphism from $ (A,B,\gamma ) $ to $ (A',B',\gamma ' ) $ is a pair $ (f:A\rightarrow A', g:B\rightarrow B') $ such that $ \gamma ' .Ff=Gg.\gamma $.

$\xymatrix{ FA \ar[d]_{Ff} \ar[r]^{\gamma } & GB \ar[d]^{Gg} \\ FA' \ar[r]_{\gamma ' } & GB' }$.

Clearly this gives a category and we have the obvious functors $ F':{\mathcal P}\rightarrow {\mathcal B} $ and $ G':{\mathcal P}\rightarrow {\mathcal A} $ with, for instance, $ G' $ taking $ (A,B,\gamma ) $ to $ A $ and taking $ (f,g) $ to $ f$.  We also have the natural isomorphism $FG'\Rightarrow GF'$ with component $\gamma$ at $(A,B,\gamma)$, which is a 2-arrow of ${\mathbb A}{\mathbb B}{\mathbb E}{\mathbb X}$.  We show that $ {\mathcal P} $ is abelian, that $ F' $ and $ G' $ are exact and that we have constructed a pullback for the diagram in $ {\mathbb A}{\mathbb B}{\mathbb E}{\mathbb X}$.

\begin{lemma}\label{abexpbkconstr} $ {\mathcal P} $ is an abelian category and $ F', G' $ are exact.
\end{lemma}
\begin{proof} We define addition of morphisms as follows: given $ (f,g), (f',g') : (A,B,\gamma )\rightarrow (A'B',\gamma ' ) $ we have the commutative diagrams

$\xymatrix{ FA \ar[d]_{Ff} \ar[r]^{\gamma } & GB \ar[d]^{Gg} \\ FA' \ar[r]_{\gamma ' } & GB' }$ and $\xymatrix{ FA \ar[d]_{Ff'} \ar[r]^{\gamma } & GB \ar[d]^{Gg'} \\ FA' \ar[r]_{\gamma ' } & GB' }$

\noindent and hence $ \gamma ' (Ff+Ff')=Gg.\gamma +Gg'.\gamma =(Gg+Gg')\gamma $, so define $ (f,g)+(f'g') $ to be $ (f+f',g+g')$. In this way we have a preadditive category.

It is easy to check that $ (A\oplus A',B\oplus B',\gamma \oplus \gamma ' ) $ is the (co)product of $ (A,B,\gamma ) $ and $ (A',B',\gamma ' )$ in ${\mathcal P}$.

Given $ (f,g):(A,B,\gamma )\rightarrow (A',B'\gamma ' ) $ we have the exact sequences $ 0\rightarrow {\rm ker}(f)\rightarrow A\xrightarrow{f} A'\rightarrow {\rm coker}(f)\rightarrow 0 $ and $ 0\rightarrow {\rm ker}(g)\rightarrow B\xrightarrow{g} B'\rightarrow {\rm coker}(g)\rightarrow 0$. Since $ F $ and $ G $ are exact these give exact sequences $ 0\rightarrow F{\rm ker}(f)={\rm ker}(Ff)\rightarrow FA\rightarrow FA' $ and similarly for $ G $ and for the cokernel sequences. Then it is easily checked that $ ({\rm ker}(f),{\rm ker}(g),\gamma_0) $ where $ \gamma_0 $ is the restriction/corestriction of $ \gamma $, is the kernel of $ (f,g) $ and the description of the cokernel is dual. In particular, $ (f,g) $ is a monomorphism, respectively epimorphism, iff both $ f $ and $ g $ are. Showing that every monomorphism is a kernel and every epimorphism a cokernel is similar. This also gives the description of exact sequences in $ {\mathcal P}$, from which it is obvious that $ F' $ and $ G' $ are exact.
\end{proof}

\begin{theorem}\label{abexpbk} $ {\mathbb A}{\mathbb B}{\mathbb E}{\mathbb X} $ has pullbacks.
\end{theorem}
\begin{proof} We show that $ {\mathcal P} $ constructed above is the pullback of the diagram we started with. So suppose that we have a diagram

$\xymatrix{{\mathcal D} \ar@/_/[ddr]_K \ar@/^/[drr]_H  & & \\ & {\mathcal P} \ar[r]^{G'} \ar[d]_{F'} & {\mathcal A} \ar[d]^{F} \\ & {\mathcal B} \ar[r]^{G} & {\mathcal C}}$

\noindent and a natural isomorphism $ \tau :FH\rightarrow GK$. Then we define $ L:{\mathcal D}\rightarrow {\mathcal P} $ by taking an object $ D $ to $ (HD,KD, \tau _D) $ and taking a morphism $ f:D\rightarrow D' $ to $ (Hf,Kf)$, noting that the diagram

$\xymatrix{ FHD \ar[r]^{\tau _D} \ar[d]_{FHf} & GKD \ar[d]^{GKf} \\ FHD' \ar[r]_{\tau _{D'}} & GKD' }$

\noindent commutes.  It is quickly checked that this functor fills in the diagram (with even strictly commuting triangles) and its uniqueness (to natural isomorphism) follows using the definitions of $ F' $ and $ G'$.  More precisely, if we also have $L':{\mathcal D} \rightarrow {\mathcal P}$ and natural isomorphisms $\eta:H\Rightarrow G'L'$, $\eta':K\Rightarrow F'L'$ then for $D\in {\mathcal D}$, $L'D$ has the form $(G'L'D, F'L'D, \delta_D)$ where $\delta_D= G\eta'_D \tau_D (F\eta_D)^{-1}$ (let us fix, because we actually have a cone over the whole diagram, the arrow from ${\mathcal D}$ to ${\mathcal C}$ to be $GK$).  So we have a natural isomorphism from $L$ to $L'$ with component at $D$ taking $LD=(HD, KD, \tau_D)$ to $L'D=(G'L'D, F'L'D, \delta_D)$ by the pair $(\eta_D, \eta'_D)$.

In the 2-categorical context, the notion of pullback (and more general limits) usually requires more, namely we must show that if we have two such cones at ${\mathcal D}$ over the initial diagram, say

$\xymatrix{ {\mathcal D} \ar[r]^{H} \ar[d]_{K} & {\mathcal A} \ar[d]^{F} \\ {\mathcal B} \ar[r]_{G} & {\mathcal C} }$ with natural isomorphism $\tau:FH\Rightarrow GK$

and  $\xymatrix{ {\mathcal D} \ar[r]^{H'} \ar[d]_{K'} & {\mathcal A} \ar[d]^{F} \\ {\mathcal B} \ar[r]_{G} & {\mathcal C} }$ with natural isomorphism $\tau':FH'\Rightarrow GK'$,

\noindent which are pullbacks in this sense, hence with natural isomorphisms $\eta:H\Rightarrow G'L$, $\eta':K\Rightarrow F'L$ and $\zeta:H'\Rightarrow G'L'$, $\zeta':K'\Rightarrow F'L'$, and if we have a modification $m$ (see, e.g.~\cite[B1.1.2]{Joh1}) from the first to the second, that is, natural transformations $m_{\mathcal A}:H\Rightarrow H'$ and $m_{\mathcal B}: K\Rightarrow K'$ such that $\tau'\cdot Fm_{\mathcal A} = m_{\mathcal B}G \cdot \tau$ ($\ast$),

$\xymatrix{ {\mathcal D} \ar[r]^{H} & {\mathcal A} \ar@{=>}[d]^{Fm_{\mathcal A}} \ar[r]^{F} & {\mathcal C} \\ {\mathcal D} \ar[r]^{H'} & {\mathcal A} \ar@{=>}[d]^{\tau'} \ar[r]^{F} & {\mathcal C} \\ {\mathcal D} \ar[r]^{K'} & {\mathcal B}  \ar[r]^{G} & {\mathcal C} }$ = $\xymatrix{ {\mathcal D} \ar[r]^{H} & {\mathcal A} \ar@{=>}[d]^{\tau} \ar[r]^{F} & {\mathcal C} \\ {\mathcal D} \ar[r]^{K} & {\mathcal B} \ar@{=>}[d]^{m_{\mathcal B}G} \ar[r]^{G} & {\mathcal C} \\ {\mathcal D} \ar[r]^{K'} & {\mathcal B}  \ar[r]^{G} & {\mathcal C} }$,

\noindent then there is a natural transformation $e:L\Rightarrow L': {\mathcal D} \rightarrow {\mathcal P}$ such that the following pairs of natural transformations are equal:

$\xymatrix{ {\mathcal D}  \ar[rr]^{H} & \ar@{=>}[d]^{\eta} & {\mathcal A} \\ {\mathcal D} \ar[r]^{L} & {\mathcal P} \ar@{=>}[d]^{G'e} \ar[r]^{G'} & {\mathcal A} \\ {\mathcal D'} \ar[r]^{L'} & {\mathcal P} \ar[r]^{G'} & {\mathcal A} }$
= $\xymatrix{ {\mathcal D}  \ar[rr]^{H} & \ar@{=>}[d]^{m_{\mathcal A}} &{\mathcal A} \\ {\mathcal D'}  \ar[rr]^{H'} & \ar@{=>}[d]^{\zeta} & {\mathcal A} \\ {\mathcal D'} \ar[r]^{L'} & {\mathcal P} \ar[r]^{G'} & {\mathcal A} }$

$\xymatrix{ {\mathcal D}  \ar[rr]^{K} & \ar@{=>}[d]^{\eta'} & {\mathcal B} \\ {\mathcal D} \ar[r]^{L} & {\mathcal P} \ar@{=>}[d]^{F'e} \ar[r]^{F'} & {\mathcal B} \\ {\mathcal D'} \ar[r]^{L'} & {\mathcal P} \ar[r]^{F'} & {\mathcal B} }$ =
$\xymatrix{ {\mathcal D}  \ar[rr]^{K} & \ar@{=>}[d]^{m_{\mathcal B}} & {\mathcal B} \\ {\mathcal D'}  \ar[rr]^{K'} & \ar@{=>}[d]^{\zeta'} & {\mathcal B} \\ {\mathcal D'} \ar[r]^{L'} & {\mathcal P} \ar[r]^{F'} & {\mathcal B} }$.

Note that $L:{\mathcal D} \rightarrow {\mathcal P}$ takes an object $D$ of ${\mathcal D}$ to $(\eta_D HD, \eta'_D KD, \eta'_D\tau_D \eta_D^{-1})$ and $L'D = (\zeta_D H'D, \zeta'_D K'D, \zeta'_D\tau'_D \zeta_D^{-1})$.  Also note that the data $m_{\mathcal A}$ at $D$ is $(m_{\mathcal A})_D: HD \rightarrow H'D$ and similarly for $m_{\mathcal B}$.  We define $e$ at $D$ to be the morphism $e_D:(\eta_D HD, \eta'_D KD, \eta'_D\tau_D \eta_D^{-1})  \rightarrow (\zeta_D H'D, \zeta'_D K'D, \zeta'_D\tau'_D \zeta_D^{-1})$ of ${\mathcal P}$ with components $\big( \zeta_D (m_{\mathcal A})_D \eta_D^{-1}, \zeta'_D (m_{\mathcal B})_D (\eta'_D)^{-1}\big)$.  By ($*$) this is indeed an arrow of ${\mathcal P}$ and one can verify that it does satisfy the required conditions.
\end{proof}

\begin{example} Let $k$ be a field and consider the two representation embeddings $F_1,F_2:{\rm mod}\mbox{-}k[T] \rightarrow {\rm mod}\mbox{-}k\widetilde{A}_1$ from $k[T]$-modules to representations of the Kronecker quiver $\xymatrix{\bullet \ar@/^/[r] \ar@/_/[r] & \bullet}$, where $F_1$ takes a $k[T]$-module, thought of as a $k$-vectorspace $V$ with a linear transformation $T_V$, to the Kronecker quiver representation $\xymatrix{V \ar@/^/[r]^{T_V} \ar@/_/[r]_{1_V} & V}$ and where the image of $(V,T_V)$ under $F_2$ is $\xymatrix{V \ar@/^/[r]^{1_V} \ar@/_/[r]_{T_V}& V}$.  The actions of $F_1$ and $F_2$ on morphisms are the obvious ones.  We compute the pullback of these two functors.

An object of this pullback has the form $((V,T_V), (W,T_W), \gamma)$ where $\gamma$ is given by a pair of isomorphisms $(\gamma_1,\gamma_2)$ making the following diagrams commute.

$\xymatrix{V \ar[r]^{T_V} \ar[d]_{\gamma_1} & V \ar[d]^{\gamma_2} \\ W \ar[r]_{1_W} & W}$
$\xymatrix{V \ar[r]^{1_V} \ar[d]_{\gamma_1} & V \ar[d]^{\gamma_2} \\ W \ar[r]_{T_W} & W}$

So $T_V$ must be an isomorphism with $T_W =\gamma_2 T_V^{-1}\gamma_2^{-1}$ being its inverse up to a change of basis.  After describing the morphisms it is easily checked that the pullback category ${\mathcal K}$ is equivalent to the category ${\rm mod}\mbox{-}k[T,T^{-1}]$, which is what one would reasonably expect it to be.
\end{example}

\subsection{$ {\mathbb A}{\mathbb B}{\mathbb E}{\mathbb X} $ is finitely accessible}\label{secabexfinacc}

We show that $ {\mathbb A}{\mathbb B}{\mathbb E}{\mathbb X} $ is, in some 2-category sense, finitely accessible.  That is, we show that there is a set of finitely presented objects such that each object of $ {\mathbb A}{\mathbb B}{\mathbb E}{\mathbb X} $ is a directed colimit of copies of these objects.  There is a variety of notions of directed colimit in the 2-categorical context, depending on the level at which diagrams are required to commute, as opposed to commute up to natural isomorphisms and so our statement refers to a particular interpretation of the words ``finitely accessible".

A {\bf directed system}, see \cite[B1.1.6]{Joh1} or \cite[\S5.4.2]{MakPar}, in $ {\mathbb A}{\mathbb B}{\mathbb E}{\mathbb X} $ is the data $\big( ({\mathcal A}_\lambda)_{\lambda\in \Lambda}, (f_{\lambda \mu }:A_\lambda \rightarrow A_\mu)_{\lambda <\mu}, (\phi_{\lambda \mu \nu}:f_{\mu\nu}f_{\lambda\mu}\Rightarrow f_{\lambda\nu})_{\lambda<\mu <\nu}\big)$ where each $\phi_{\lambda\mu\nu}$ is a natural isomorphism and where we assume the coherence condition $\phi_{\lambda\nu\rho} \phi_{\lambda\mu\nu} = \phi_{\lambda\mu\rho} \phi_{\mu\nu\rho}$.

$
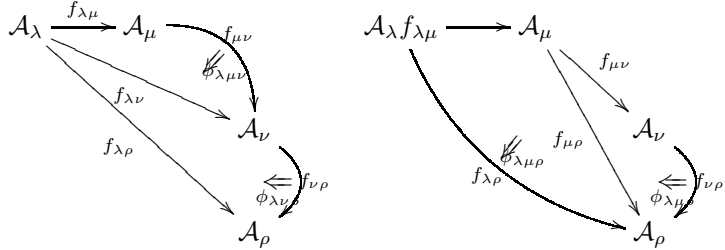
\xymatrix{{\mathcal A}_\lambda \ar[r]^{f_{\lambda \mu}} \ar[rrd]_{f_{\lambda \nu}} \ar[rrdd]_{f_{\lambda \rho}} & {\mathcal A}_\mu \druppertwocell^{f_{\mu \nu}}{\phi_{\lambda \mu \nu}} \\
& & {\mathcal A}_\nu \duppertwocell^{f_{\nu \rho}}{\phi_{\lambda \nu \rho}} \\ & & {\mathcal A}_\rho} $
$\xymatrix{{\mathcal A}_\lambda \ar[r]{f_{\lambda \mu}} \ddrrlowertwocell_{f_{\lambda \rho}}{\phi_{\lambda \mu \rho}} & {\mathcal A}_\mu \ar[dr]^{f_{\mu \nu}} \ar[rdd]_{f_{\mu \rho}} \\
& & {\mathcal A}_\nu \duppertwocell^{f_{\nu \rho}}{\phi_{\lambda \mu \rho}} \\ & & {\mathcal A}_\rho}$

There is a corresponding notion of directed colimit which, cf.~the proof of \ref{abexpbk}, has a clause involving modifications.  In the next result we use that but it will be seen that, for our main result, we are able to simplify matters by using strictly directed diagrams.

By a {\bf cocone} on such a directed system we mean an object ${\mathcal A}$, arrows $(f_{\lambda \infty}:{\mathcal A}_\lambda \rightarrow {\mathcal A})_\lambda$ and natural isomorphisms $(\theta_{\lambda \mu}: f_{\mu \infty}f_{\lambda \mu} \Rightarrow f_{\lambda \infty}\big)_{\lambda <\mu}$ such that $\theta_{\lambda \mu}\theta_{\mu \nu} = \theta_{\lambda \nu}$.  This is a {\bf direct limit} (or {\bf directed colimit}) of the system if, given any cocone $\big({\mathcal B}, (g_{\lambda}:{\mathcal A}_\lambda \rightarrow {\mathcal B})_\lambda, \zeta_{\lambda \mu}: g_{\mu }f_{\lambda \mu} \Rightarrow g_{\lambda}\big)$, there is $h:{\mathcal A} \rightarrow {\mathcal B}$ and there are natural isomorphisms $\eta_\lambda: hf_{\lambda \infty} \Rightarrow g_\lambda$ with $\eta_\lambda \theta_{\lambda \mu} = \zeta_{\lambda \mu}\eta_\mu$ for all $\lambda < \mu$, plus a clause involving modifications which we will not spell out here.

\begin{prop}\label{abexlims} The 2-category $ {\mathbb A}{\mathbb B}{\mathbb E}{\mathbb X} $ has colimits of weakly directed diagrams.
\end{prop}
\begin{proof} Suppose we are given a weakly directed system $\big( ({\mathcal A}_\lambda)_{\lambda\in \Lambda}, (f_{\lambda \mu }:A_\lambda \rightarrow A_\mu)_{\lambda <\mu}, (\phi_{\lambda \mu \nu}:f_{\mu\nu}f_{\lambda\mu}\Rightarrow f_{\lambda\nu})_{\lambda<\mu <\nu}\big)$ as above.  We define the category ${\mathcal A}$ which will be the colimit, as follows.

For the objects of ${\mathcal A}$ we take the equivalence classes of objects of $\bigcup_\lambda {\mathcal A}_\lambda$ under the equivalence relation $\sim$ generated by setting $A_\lambda \sim f_{\lambda \mu}A_\lambda$.  So this identifies $f_{\lambda\nu} A_\lambda$ and $f_{\mu \nu}f_{\lambda \mu}A_\lambda$ whenever $\lambda < \mu < \nu$.  Therefore $A_\lambda \sim A_\mu$ iff there are $\lambda_0=\lambda, \lambda_1, \dots, \lambda_n=\mu$ and objects $A_{\lambda_0}=A_\lambda, A_{\lambda_i} \in {\mathcal A}_{\lambda_i}, \dots, A_{\lambda_n}=A_\mu$ with, for each $i$, either $\lambda_i <\lambda_{i+1}$ or $\lambda_{i+1}<\lambda_i$ and correspondingly $A_{\lambda_{i+1}}=f_{\lambda_i \lambda_{i+1}}A_{\lambda_i}$ or $A_{\lambda_i} = f_{\lambda_{i+1} \lambda_i}A_{\lambda_{i+1}}$.  Note that in this case if $\nu >\lambda_0,\dots, \lambda_n$ then the objects $f_{\lambda \nu}A_\lambda$ and $f_{\mu \nu}A_\mu$ are objects in ${\mathcal A}_\nu$ which are isomorphic by a sequence of components of some of the $\phi_{ijk}$ and their inverses (depending on a choice of ``zig-zag'' between $A_\lambda$ and $A_\mu$).  We continue to use the obvious subscript notation to indicate which category an object lies in.

Similarly we define an equivalence relation $\sim$ on arrows to be that generated by setting $g:A_\lambda \rightarrow B_\lambda$ to be equivalent to $ f_{\lambda \mu}g$ for all $\mu > \lambda$ and by setting $1_{A_\lambda}$ to be equivalent to the component of $\phi_{\lambda \mu \nu}$ at ${A_\lambda}$ for all $A_\lambda$ and $\lambda<\mu <\nu$.  The arrows from $A_\lambda/\sim$ to $B_\nu/\sim$ are the equivalence classes of arrows from some $A_\mu \in A_\lambda/\sim$ to some $B_\mu \in B_\nu/\sim$.

To define composition of such arrows it is sufficient to consider the case of arrows $g:A_\lambda \rightarrow B_\lambda$ and $h:B_\mu \rightarrow C_\mu$ when $B_\lambda \sim B_\mu$.  Using notation as above for a zig-zag between $B_\lambda$ and $B_\mu$, choose $\nu>\lambda_0,\dots, \lambda_n$ and note that the objects $f_{\lambda \nu}B_\lambda$ and $f_{\mu\nu}B_\mu$ are connected by a sequence, $k$ say, of forward images of components of the $\phi_{\rho \sigma \tau}$ and their inverses - arrows in the equivalence class of the identity map of the object $B_\lambda/\sim$.  We define the composition to be $f_{\mu\nu}h .k. f_{\lambda \nu}g$.

It has to be checked that the result is abelian, that the obvious $f_{\lambda \infty}:{\mathcal A}_\lambda \rightarrow {\mathcal A}$ are exact (that is clear), and that ${\mathcal A}$ has the universal property including the modifications clause.  Checking that ${\mathcal A}$ is abelian may be done by using that, given any finite diagram in ${\mathcal A}$, one may choose representatives of the objects and arrows in it and then find a single ${\mathcal A}_\nu$ in which there is an actual diagram of the same sort which represents the original one.  The directed colimit property may be shown by arguing rather as in the proof of \ref{limofabs} below.
\end{proof}

Although we use a weak notion of directed system and directed colimit in $ {\mathbb A}{\mathbb B}{\mathbb E}{\mathbb X} $ we will now see that if we have a directed colimit in the category ${\mathcal Rng}$ of rings then this induces a diagram of associated free abelian categories which even has strictly commuting compositions, that is, with $f_{\mu \nu}f_{\lambda \nu} =f_{\lambda \nu}$, hence with each $\phi_{\mu \nu\lambda}$ being the identity.  We do this by choosing specific copies of free abelian categories, as follows (this will refer to the connections with model theory, for which see, e.g., \cite{PreNBK}).

Suppose that $f:R\rightarrow S$ is a morphism of rings.  If $\phi$ is a formula (treated as a string of symbols) in the language of $R$-modules then we define $f_\ast \phi$ to be the formula in the language of $S$-modules which is obtained by replacing each occurrence of (the function symbol corresponding to) an element of $R$ by (the function symbol corresponding to) its image in $S$.  It is a result of Burke (\cite[3.2.5]{BurThes}; see, e.g., \cite[10.2.30]{PreNBK}) that the free abelian category can be regarded (up to equivalence of categories) as having objects the pairs, $\phi/\psi$, of pp formulas and having equivalence classes of certain pp formulas as its morphisms.  Clearly $f_\ast$ immediately gives a map from the objects of ${\rm Ab}(R)$, so defined, to ${\rm Ab}(S)$; it also gives a map on morphisms - given a morphism in ${\rm Ab}(R)$ from $\phi /\psi$ to $\phi'/\psi'$ we choose a representative formula $\rho$ which defines it and then, one may check, $f_\ast \rho$ defines a map from $f_\ast \phi /f_\ast \psi$ to $f_\ast \phi'/f_\ast \psi'$ which, again, one may check (for, if formulas are equivalent on $R$-modules then they are equivalent on $S$-modules regarded as $R$-modules), is independent of choice of representing formula $\rho$.  In this way $f$, through $f_\ast$, induces what is clearly a functor, let us denote it ${\rm Ab}(f)$, from ${\rm Ab}(R)$ to ${\rm Ab}(S)$.  We said this for rings but all of it applies equally for small preadditive categories ${\mathcal R}$, ${\mathcal S}$ in place of $R$ and $S$.

The point of this construction is that if we are also given a homomorphism $g:S\rightarrow T$ then ${\rm Ab}(g){\rm Ab}(f) ={\rm Ab}(gf)$ - equality, not natural isomorphism.  Therefore, given a (directed) diagram $\Delta$ in ${\mathcal Rng}$, we have a (directed) diagram ``${\rm Ab}(\Delta)$" in $ {\mathbb A}{\mathbb B}{\mathbb E}{\mathbb X} $ where any commutativity relations in the original diagram $\Delta$ are replaced by strict commutativity relations between the corresponding functors; we will use the term {\bf strictly directed} for such directed diagrams.  This will allow us, at least for our considerations, to keep things simple, though in a way which, no doubt, is not actually necessary.

\begin{prop}\label{limofabs} Suppose that $ R=\varinjlim R_\lambda  $ in $ {\mathcal Rng}$, more generally, suppose $ {\mathcal R}=\varinjlim {\mathcal R}_\lambda  $ is a strict colimit of a strictly directed diagram in $ {\mathbb P}{\mathbb R}{\mathbb E}{\mathbb A}{\mathbb D}{\mathbb D} $. Then $ {\rm Ab}({\mathcal R})=\varinjlim {\rm Ab}({\mathcal R}_\lambda ) $ in $ {\mathbb A}{\mathbb B}{\mathbb E}{\mathbb X}$ (and, if appropriate choices of copies of the ${\rm Ab}({\mathcal R}_\lambda)$ are made, this may be taken to be a strict direct limit).
\end{prop}
\begin{proof} (We prove it for rings, the modifications for the more general case being minor.  Also, to keep the notation natural, we write $R$ for the single object of the ring $R$ regarded as a 1-point category.) The $f_{\lambda \mu}:R_\lambda \rightarrow R_\mu$ induce ${\rm Ab}(f_{\lambda \mu}):{\rm Ab}(R_\lambda) \rightarrow {\rm Ab}(R_\mu)$ which, as we have seen above, may be taken to form a strict directed system in ${\mathbb A}{\mathbb B}{\mathbb E}{\mathbb X}$.  Similarly we may take a strictly commuting cocone on this diagram formed by the ${\rm Ab}(f_{\lambda \infty}): {\rm Ab}(R_\lambda) \rightarrow {\rm Ab}(R)$, where $f_{\lambda \infty}:R_\lambda \rightarrow R$ are the maps in ${\mathcal Rng}$ to the direct limit.  We claim that this is a directed colimit in ${\mathbb A}{\mathbb B}{\mathbb E}{\mathbb X}$.

So suppose that we have ${\mathcal B}\in {\mathbb A}{\mathbb B}{\mathbb E}{\mathbb X}$ and for each $\lambda$ an exact functor $g_\lambda: {\rm Ab}(R_\lambda) \rightarrow {\mathcal B}$ such that for each $\mu > \lambda$ there is a natural isomorphism $\rho_{\lambda\mu}: g_\mu {\rm Ab}(f_{\lambda \mu}) \rightarrow g_\lambda$ such that these cohere to make a cocone on the ${\rm Ab}(-)$ diagram; in particular $\rho_{\mu \nu} \rho_{\lambda \mu} = \rho_{\lambda \nu}$.  The canonical embeddings $R_\lambda \rightarrow {\rm Ab}(R_\lambda)$, when composed with the $g_\lambda$, give objects $g_\lambda R_\lambda$ of ${\mathcal B}$ which are linked by the strictly directed system of isomorphisms $((\rho_{\lambda \mu})_{(R_\lambda)})^{-1}:g_\lambda R_\lambda \rightarrow g_\mu R_\mu$.  We choose and fix some $\lambda\in \Lambda$, set $S=g_\lambda R_\lambda$ and note that the $(\rho_{\lambda \mu})_{R_\lambda} (g_\mu \upharpoonright_{R_\mu})$ form a (strictly commuting) cocone in ${\mathcal Rng}$ from the family $(R_\mu)_{\mu\geq \lambda}$ to $S$, where $g_\mu \upharpoonright_{R_\mu}$ denotes the ring homomorphism from $R_\mu$ to $g_\mu R_\mu$ induced by $g_\mu$.  Hence there is an induced morphism, $h_0$, from $R$ to $S$.  We have to extend this to an exact functor $h:{\rm Ab}(R) \rightarrow {\mathcal B}$ which will form a cone (in the 2-category sense) on the directed system $({\rm Ab}(R_\mu))_\mu$.

We note that the images of the $\rho_{\lambda \mu}$ form a strictly cohering (as $\mu$ varies) system of objects and arrows of ${\mathcal B}$ which, together, form a copy of ${\rm Ab}(S)$ and which, by our previous specific construction of ${\rm Ab}(R)$ as a direct limit of the ${\rm Ab}(R_\mu)$, induce, {\it via} the $g_\mu$ (then the $\rho_{\lambda \mu}$) an (exact) functor $h$, extending $h_0$, from ${\rm Ab}(R)$ to this copy of ${\rm Ab}(S)$ and thence to ${\mathcal B}$.  This functor $h$ is, by this construction (admittedly, hardly in the spirit of 2-category theory), such that $h {\rm Ab}(f_{\mu \infty})=g_\lambda$ and so we obtain the statement of the theorem.
\end{proof}

\begin{cor}\label{ablimfinite} Given a skeletally small preadditive category $ {\mathcal R} $ let $ \{ {\mathcal R}_\lambda \}_\lambda  $ be the directed system of its full subcategories with finitely many objects. Then $ {\rm Ab}({\mathcal R})=\varinjlim {\rm Ab}({\mathcal R}_\lambda )$. If $ {\mathcal S}\in {\rm Ser}({\rm Ab}({\mathcal R})) $ then $ {\rm Ab}({\mathcal R})/{\mathcal S}=\varinjlim {\rm Ab}({\mathcal R}_\lambda )/{\mathcal S}_\lambda  $ where $ {\mathcal S}_\lambda ={\mathcal S}\cap {\rm Ab}({\mathcal R}_\lambda )$.
\end{cor}
\begin{proof} First replace ${\mathcal R}$ by a small category to which it is equivalent, so that we have a directed system over a set.  Then note that this directed system in ${\mathbb P}{\mathbb R}{\mathbb E}{\mathbb A}{\mathbb D}{\mathbb D} $  is strictly directed (by inclusions) and has ${\mathcal R}$ as its strict direct limit so, by \ref{limofabs}, we have the first assertion.  Similarly, and making use of the observations before \ref{limofabs}, we obtain the second statement.
\end{proof}

\begin{lemma}\label{abisquot} Each $ {\mathcal A}\in {\mathbb A}{\mathbb B}{\mathbb E}{\mathbb X} $ is equivalent to $ {\rm Ab}({\mathcal R})/{\mathcal S} $ for some small preadditive $ {\mathcal R} $ and some Serre subcategory $ {\mathcal S} $ of $ {\rm Ab}({\mathcal R})$.
\end{lemma}
\begin{proof} To see this we may, for example, set $ {\mathcal D}={\rm Ex}({\mathcal A},{\bf Ab}) $ - a definable category, hence a definable subcategory of some $ {\rm Mod}\mbox{-}{\mathcal R}$. Then $ {\mathcal A} $ is a quotient of $ {\rm fun}({\rm Mod}\mbox{-}{\mathcal R})={\rm Ab}({\mathcal R})^{\rm op} $ by some Serre subcategory.
\end{proof}

\begin{cor}\label{abaslim} If $ {\mathcal A}\in {\mathbb A}{\mathbb B}{\mathbb E}{\mathbb X} $ then $ {\mathcal A}\simeq \varinjlim {\mathcal A}_\lambda  $ where each category $ {\mathcal A}_\lambda  $ is equivalent to one the form $ {\rm Ab}(R_\lambda )/{\mathcal S}_\lambda  $ with $ R_\lambda  $ a ring and $ {\mathcal S}_\lambda \in {\rm Ser}({\rm Ab}(R_\lambda ))$.
\end{cor}
\begin{proof} By \ref{abisquot}, $ {\mathcal A}$ is equivalent to some ${\rm Ab}({\mathcal R})/{\mathcal S} $ and, by \ref{ablimfinite}, $ {\rm Ab}({\mathcal R})=\varinjlim {\rm Ab}({\mathcal R}_\lambda ) $ where each $ {\mathcal R}_\lambda  $ has just finitely many objects and hence is equivalent to a ring $ R_\lambda  $ (with $ {\rm Ab}({\mathcal R}_\lambda )\simeq {\rm Ab}(R_\lambda )$).
\end{proof}

When using this result we will typically write $ {\mathcal A} = \varinjlim {\mathcal A}_\lambda  $ in line with our use of ``$=$" between categories to mean naturally equivalent.

We will say that a category ${\mathcal A} \in {\mathbb A}{\mathbb B}{\mathbb E}{\mathbb X}$ is {\bf finitely presented} if whenever we have a directed system $\big( ({\mathcal B}_\lambda)_{\lambda\in \Lambda}, (g_{\lambda \mu }:B_\lambda \rightarrow B_\mu)_{\lambda <\mu}, (\phi_{\lambda \mu \nu}:g_{\mu\nu}g_{\lambda\mu}\Rightarrow g_{\lambda\nu})_{\lambda<\mu <\nu}\big)$ with direct limit $\big({\mathcal B}, (g_{\lambda \infty}:{\mathcal B}_\lambda \rightarrow {\mathcal B})_\lambda, \theta_{\lambda \mu}: g_{\mu \infty}g_{\lambda \mu} \Rightarrow g_{\lambda \infty}\big)$ and an exact functor $h:{\mathcal A} \rightarrow {\mathcal B}$ there is a factorisation through the directed system in the sense that there is $\lambda$, $h_\lambda :{\mathcal A} \rightarrow {\mathcal B}_\lambda$ and a natural isomorphism $\eta: h \Rightarrow g_{\lambda \infty}h_\lambda$ for some some $\lambda$.  Note that this then induces a family of arrows $h_\mu =g_{\lambda \mu}h_\lambda : {\mathcal A} \rightarrow {\mathcal B}_\mu$ and a coherent (using the $\phi_{\lambda \mu \nu}$) family $\theta_{\lambda \mu}^{-1} \eta: h\Rightarrow g_{\mu \infty} h_\mu$ of natural equivalences.  Write ${\mathbb A}{\mathbb B}{\mathbb E}{\mathbb X}^{\rm fp}$ for the full sub-2-category on the finitely presented objects.

\begin{theorem}\label{abfprng} (1) If $ R $ is a ring then $R$ is finitely presented iff $ {\rm Ab}(R)\in {\mathbb A}{\mathbb B}{\mathbb E}{\mathbb X}^{\rm fp}$.

\noindent (2) $ {\mathcal A}\in {\mathbb A}{\mathbb B}{\mathbb E}{\mathbb X}^{\rm fp} $ iff $ {\mathcal A}\simeq {\rm Ab}(R)/{\mathcal S} $ for some finitely presented ring $ R $ and finitely generated Serre subcategory of ${\rm Ab}({\mathcal R})$, where by that we mean that $ {\mathcal S}=\langle S\rangle  $ for some $ S\in {\rm Ab}(R)$.
\end{theorem}
\begin{proof} (1)($\Leftarrow $) Set $ R=\varinjlim R_\lambda  $ with $ R_\lambda  $ finitely presented so, by \ref{limofabs}, $ {\rm Ab}(R)=\varinjlim {\rm Ab}(R_\lambda ) $ and hence $ {\rm id}_{{\rm Ab}(R)} $ factors, up to natural equivalence, through some $ {\rm Ab}(R_\lambda )$. The distinguished objects $ R $ of $ {\rm Ab}(R ) $ and $ R_\lambda  $ of $ {\rm Ab}(R_\lambda ) $ are, therefore, isomorphic.

\noindent ($\Rightarrow $) Suppose that $\big({\mathcal B}, (g_{\lambda \infty}:{\mathcal B}_\lambda \rightarrow {\mathcal B})_\lambda, \theta_{\lambda \mu}: g_{\mu \infty}g_{\lambda \mu} \Rightarrow g_{\lambda \infty}\big)$, with $\theta_{\lambda \mu} \theta_{\mu \nu} = \theta_{\lambda \nu}$, is a directed colimit in $ {\mathbb A}{\mathbb B}{\mathbb E}{\mathbb X} $ of a system $\big( ({\mathcal B}_\lambda)_{\lambda\in \Lambda}, (g_{\lambda \mu }:B_\lambda \rightarrow B_\mu)_{\lambda <\mu}, (\phi_{\lambda \mu \nu}:g_{\mu\nu}g_{\lambda\mu}\rightarrow g_{\lambda\nu})_{\lambda<\mu <\nu}\big)$ as above and suppose that we have $ h:{\rm Ab}(R)\rightarrow {\mathcal B}$, hence, writing $R$ in place of $\ast_R$, $ R\rightarrow {\rm Ab}(R)\rightarrow {\mathcal B} $ with image $ R'$, say, in $ {\mathcal B}$.  We may (e.g.~see the proof of \ref{abexlims}) replace ${\mathcal B}$ with an equivalent subcategory where every object and arrow is in the image of some $g_{\lambda \infty}$ so we may assume that $R'$ is the image, $g_{\lambda\infty}B_0$ for some object $B_0$ of some chosen and then fixed ${\mathcal B}_\lambda$.  Define the directed system in $ {\mathcal Rng} $ (the objects being regarded as one-point preadditive categories) with index set $\{ \mu: \mu \geq \lambda\}$ and which, at $\mu>\lambda$ has the object $f_{\lambda\mu}B_0$ and, for $\mu < \nu$ has the functor/ring homomorphism $(\phi_{\lambda\mu\nu})_{B_0}g_{\mu\nu}: g_{\lambda\mu}B_0 \rightarrow g_{\lambda\nu}B_0$.  That this is a directed system in $ {\mathcal Rng} $ follows from the coherence condition on the $\phi$ in the definition of directed system.  Since $ {\mathcal B}=\varinjlim {\mathcal B}_\lambda  $, the direct limit of this system is (isomorphic to) $ R' $.  Since $ R $ is finitely presented this functor $ R\rightarrow {\mathcal B} $ therefore factors through the system at some $ f_{\lambda\mu}B_0$ and then that map induces $ h_\lambda:{\rm Ab}(R)\rightarrow {\mathcal B}_\lambda  $ which lifts $ h:{\rm Ab}(R)\rightarrow B $ through the system in the sense described in the definition of finitely presented, as required.

(2)($\Rightarrow $) The category ${\mathcal A}$ is, by \ref{abisquot}, up to equivalence, ${\rm Ab}({\mathcal R})/{\mathcal S}$ for some ${\mathcal R}$ and ${\mathcal S} \in {\rm Ser}({\rm Ab}({\mathcal R}))$.  We have $ {\mathcal S}=\varinjlim _{S\in {\mathcal S}}\langle S\rangle  $ and, correspondingly, $ {\mathcal A}={\rm Ab}({\mathcal R})/{\mathcal S}=\varinjlim {\rm Ab}({\mathcal R})/\langle S\rangle  $ so, if $ {\mathcal A} $ is finitely presented, already $ {\mathcal A}={\rm Ab}({\mathcal R})/\langle S\rangle  $ for some $ S\in {\mathcal S}$. Also (\ref{abaslim}), $ {\mathcal A}=\varinjlim {\rm Ab}(R)/{\rm Ab}(R)\cap \langle S\rangle  $ where the limit is taken over (rings equivalent to) the finite sets of objects of $ {\mathcal R}$, and so, if $ {\mathcal A} $ is finitely presented then it has the form claimed.

\noindent ($\Leftarrow $) Suppose we have $ {\mathcal A} $ as in the statement and, using brief notation, $ {\mathcal A}\rightarrow {\mathcal B}=\varinjlim {\mathcal B}_\lambda  $ a directed colimit in $ {\mathbb A}{\mathbb B}{\mathbb E}{\mathbb X}$. The composition $ R\rightarrow {\rm Ab}(R)\rightarrow {\mathcal A}={\rm Ab}(R)/{\mathcal S}\rightarrow {\mathcal B} $ lifts through some $k:{\rm Ab}(R) \rightarrow {\mathcal B}_\lambda$ as indicated, using the data of the directed system to construct a directed system of rings as in the first part of the proof and then using that $ R $ is finitely presented.

$\xymatrix{R \ar[r] \ar@{.>}[dr] & {\rm Ab}(R) \ar[r] \ar@{.>}[d]_{k} & {\mathcal A} \ar[dr] \ar@{.>}[d] \\ & {\mathcal B}_\lambda \ar[r] \ar@/_/[rr] & {\mathcal B}_\mu \ar[r] & {\mathcal B}}$

\noindent The composition of $ R\rightarrow {\mathcal B}_\lambda  $ with $ B_\lambda \rightarrow {\mathcal B} $ induces an exact functor $ {\rm Ab}({\mathcal R})\rightarrow {\mathcal B}$ which must be equivalent to the given functor from ${\rm Ab}(R)$ to ${\mathcal B}$.

Since $ {\mathcal S}=\langle S\rangle  $ and $ {\rm Ab}(R) \rightarrow {\mathcal A}\rightarrow {\mathcal B} $ sends $ S $ to $ 0$, there is $ \mu \geq \lambda  $ such that the composite $ {\mathcal S}\rightarrow {\rm Ab}(R)\rightarrow {\mathcal B}_\lambda \rightarrow {\mathcal B}_\mu  $ is $ 0$ and so there is $ {\mathcal A}\rightarrow {\mathcal B}_\mu  $ as shown making the square commute.

The composites $ R\rightarrow {\rm Ab}(R)\rightarrow {\mathcal A}\rightarrow {\mathcal B} $ and $ R\rightarrow {\rm Ab}(R)\rightarrow {\mathcal B}_\lambda \rightarrow {\mathcal B}_\mu \rightarrow {\mathcal B} $ agree up to isomorphism on $ R $ and so $ {\rm Ab}(R)\rightarrow {\mathcal A}\rightarrow {\mathcal B} $ and $ {\rm Ab}(R)\rightarrow {\mathcal B}_\mu \rightarrow {\mathcal B} $ are, by \ref{freeabthm}, naturally equivalent; both have kernel $ {\mathcal S}$, so they induce an equivalence between $ {\mathcal A}\rightarrow {\mathcal B} $ and $ {\mathcal A}\rightarrow {\mathcal B}_\mu \rightarrow {\mathcal B}$, as required.
\end{proof}

Thus $ {\mathbb A}{\mathbb B}{\mathbb E}{\mathbb X} $ has a set of objects which are finitely presented and such that every object of $ {\mathbb A}{\mathbb B}{\mathbb E}{\mathbb X} $ is equivalent to the directed colimit of a diagram composed of copies of these.

\subsection{Abelian categories as schemes}\label{secabsch}

The replacement of a ring by its free abelian category and the role that small abelian categories play, as categories of imaginaries, in the model theory of additive structures (\cite{PreNBK}, \cite{PreDefAddCat}, \cite{PreAxtFlat}), strongly suggest the heuristic that small abelian categories are a generalisation of (some aspects of) rings.  This, especially in view of the connection with Ziegler and rep-Zariski spectra (e.g., \cite{PreNBK}, \cite{PreRajShv}), in turn suggests developing some additive version of the theory of schemes with $ {\mathbb A}{\mathbb B}{\mathbb E}{\mathbb X}$ playing the role of the category of commutative rings and the category ${\bf Ab}$ replacing the category of sets in the functor-of-points approach (and with the associated presheaves of abelian categories - see, e.g.~\cite{PreRajShv} - being the analogues of varieties and schemes).

In fact, there seem to be two natural ways of embedding the category of (all unital) rings into this context.  One is {\it via} the free abelian category construction (see \ref{freeabthm}) but the resulting ``geometry'' is really that of the representation theory of the ring, rather than that of the ring itself.  An embedding that perhaps better reflects classical algebraic geometry is obtained by taking a ring $R$ to ${\mathcal A}(R)$ - the smallest abelian subcategory of ${\rm Mod}\mbox{-}R$ which contains the category of finitely presented modules; in this case, however, it seems that we should restrict ring morphisms to be the flat ones, and we do lose information, see below.

We now make the obvious definitions but must point out that the notion of ``affine'' scheme in this context is unclear, for instance since many non-affine schemes, such as the projective line can already be found as the associated geometry of a category in $ {\mathbb A}{\mathbb B}{\mathbb E}{\mathbb X}$.

\subsubsection{The functor of points view}\label{secfunpt}

Recall that if $ Z={\rm Spec}(R) $ is an affine scheme then by a scheme over $ Z $ one means a morphism $ X\rightarrow Z $ of schemes. If $ X $ is affine, say $ X={\rm Spec}(S)$, then this is equivalent to a morphism $ R\rightarrow S $ in the category of (commutative) rings. If we are thinking of $ {\mathbb A}{\mathbb B}{\mathbb E}{\mathbb X} $ as a generalisation of the category of commutative rings and if we assume that any reasonable embedding of the latter category in the former is covariant, then it seems reasonable to say that a {\bf scheme over} $ {\mathcal A}\in {\mathbb A}{\mathbb B}{\mathbb E}{\mathbb X} $ is a morphism $ f:{\mathcal A}\rightarrow {\mathcal B} $ in $ {\mathbb A}{\mathbb B}{\mathbb E}{\mathbb X}$. By the anti-equivalence between $ {\mathbb A}{\mathbb B}{\mathbb E}{\mathbb X} $ and $ {\mathbb D}{\mathbb E}{\mathbb F} $ we can extend the terminology, by saying that an {\bf scheme over} $ {\mathcal C}={\rm Ex}({\mathcal A},{\bf Ab}) $ is a morphism $ {\mathcal D}={\rm Ex}({\mathcal B},{\bf Ab})\rightarrow {\mathcal C}$. In model-theoretic terms this is a coherent way of interpreting in each $ D\in {\mathcal D} $ some $ {\mathcal C}$-structure.

For example, a morphism $ f:R\rightarrow S $ of rings induces an exact functor $ {\rm Ab}(f):{\rm Ab}(R)\rightarrow {\rm Ab}(S) $ and the corresponding morphism of definable categories is just the induced restriction-of-scalars functor ${\rm Mod}\mbox{-}S\rightarrow {\rm Mod}\mbox{-}R $ which clearly fits the description of being a coherent way of interpreting an $ R$-module in each $ S$-module.

An alternative view of a morphism $ X\rightarrow Z $ of schemes is that it is an $ X$-point of $ Z$, so we may extend that terminology also, saying that a morphism $ {\mathcal A}\rightarrow {\mathcal B} $ in $ {\mathbb A}{\mathbb B}{\mathbb E}{\mathbb X} $ is, as well as a scheme over $ {\mathcal A}$, a $ {\mathcal B}$-{\bf point} of $ {\mathcal A}$. The collection of $ {\mathcal B}$-points of $ {\mathcal A}$, being just $ {\rm Ex}({\mathcal B},{\mathcal A})$, has a natural structure of an additive category.

A particularly important case is where $ {\mathcal A}={\rm Ab}({\mathbb Z}) $ is the free abelian category of $ {\mathbb Z}$. A scheme over $ {\rm Ab}({\mathbb Z}) $ is an exact functor $ {\rm Ab}({\mathbb Z})\rightarrow {\mathcal B} $ to a small abelian category. By \ref{freeabthm} and since $ {\mathbb Z} $ has just the one object, and since there is no choice about where to send the endomorphisms of that object, a $ {\mathcal B}$-point of $ {\rm Ab}({\mathbb Z}) $ is simply an object of $ {\mathcal B}$. The corresponding morphisms of definable categories are just the functors in ${\mathbb D}{\mathbb E}{\mathbb F}$ from $ {\mathcal D}={\rm Ex}({\mathcal B},{\bf Ab}) $ to $ {\rm Ex}({\rm Ab}({\mathbb Z}),{\bf Ab})={\bf Ab} $ and we know that the category of these is $ {\rm fun}({\mathcal D})={\mathcal B}$. That is, a scheme over $ {\rm Ab}({\mathbb Z}) $ is just a pp-pair (in the theory of some definable additive category) and the collection of $ {\mathcal B}$-points of $ {\rm Ab}({\mathbb Z}) $ is the category $ {\mathcal B} $ itself (in this case an abelian category though we have already noted, \ref{notab}, that $ {\rm Ex}({\mathcal A},{\mathcal B}) $ is not in general abelian).

This also suggests the view that fixing $ {\mathcal A} $ is fixing an abelian language and then a functor from a definable category to $ {\rm Ex}({\mathcal A},{\bf Ab}) $ is a generalised pp formula, picking out not just a single abelian group (and by implication its associated imaginary sorts) but a fixed collection of sorts and maps between them (and their associated sorts). Then a morphism $ {\mathcal D}= {\rm Ex}({\mathcal B},{\bf Ab}) \rightarrow {\mathcal C}= {\rm Ex}({\mathcal A},{\bf Ab}) $ in ${\mathbb D}{\mathbb E}{\mathbb F}$ is an interpretation of an (exact) $ {\mathcal A}$-structure in each (exact) $ {\mathcal B}$-structure.

\subsubsection{Rings to Abelian Categories}\label{secrngtoab}

We describe two ways of embedding the category $ {\mathcal Rng} $ of rings, more generally the 2-category $ {\mathbb P}{\mathbb R}{\mathbb E}{\mathbb A}{\mathbb D}{\mathbb D}$ of skeletally small preadditive categories, into ${\mathbb A}{\mathbb B}{\mathbb E}{\mathbb X}$.  For simplicity, we will deal just with rings (usually we think of ${\mathcal Rng} $ as an ordinary category but if $f,g:R\rightarrow S$ are homomorphisms of rings then a natural transformation/2-arrow from $f$ to $g$ is just a homomorphism from $S$ regarded as an $R$-module {\it via} $f$ to $S$ regarded as an $R$-module {\it via} $g$.)

The first is the free-abelian weak 2-functor $ {\rm Ab}(-)$, which takes $ {\mathcal R}\in {\mathbb P}{\mathbb R}{\mathbb E}{\mathbb A}{\mathbb D}{\mathbb D}$, to its free abelian category $ {\rm Ab}({\mathcal R}) $ (composing with duality also gives the contravariant version, taking $ {\mathcal R} $ to $ {\rm Ab}({\mathcal R})^{\rm op}$).  The corresponding definable category is $ {\rm Ex}({\rm Ab}(R)^{({\rm op})},{\bf Ab}) \simeq R^{({\rm op})}\mbox{-}{\rm Mod} $ whose (rep-Zariski) spectrum is in general much larger than any notion of $ {\rm Spec}(R)$. So, while that embedding is a natural one, it seems not to be the geometrically natural one (unless we restrict to, say, the subcategory of injectives).  Nevertheless, it is there and probably should be seen as describing an embedding of the representation theory of a ring rather than the geometry of a ring.

\begin{lemma}\label{evalatM}\marginpar{evalatM} If $ M\in R\mbox{-}{\rm Mod} $ then the exact functor it induces from $ {\rm Ab}(R) $ to $ {\bf Ab} $ by virtue of commutativity of the diagram shown

$\xymatrix{R \ar[rr] \ar[dr] & & {\rm Ab}(R) \ar@{.>}[dl]^{{\rm ev}_M } \\ & {\bf Ab}}$

\noindent  is $ F\in (R\mbox{-}{\rm mod},{\bf Ab})^{\rm fp} \mapsto \overline{F}M$ where the latter is the extension of $F$ to a functor on $R\mbox{-}{\rm Mod}$ which commutes with direct limits. We denote this by $ {\rm ev}_M $.
\end{lemma}

\begin{cor}\label{abfrng}\marginpar{abfrng} Any morphism $ f:R\rightarrow S $ of rings induces a unique-to-natural-equivalence (or even unique, see Section \ref{secabexfinacc}) exact functor ${\rm Ab}(f):{\rm Ab}(R)\rightarrow {\rm Ab}(S) $ as shown.

$\xymatrix{R \ar[r] \ar[d]_f & {\rm Ab}(R) \ar@{.>}[d]^{{\rm Ab}(f)} \\ S \ar[r] & {\rm Ab}(S)}$

If $ F\in {\rm Ab}(R)=(R\mbox{-}{\rm mod},{\bf Ab})^{\rm fp} $ then $ {\rm Ab}(f)F $ is defined as a functor from $ S\mbox{-}{\rm mod} $ to $ {\bf Ab} $ by $ {\rm Ab}(f)F.(_SM)=\overline{F}(_RM)$.
\end{cor}

The induced effect on modules, that is, the restriction-of-scalars functor from ${\rm Mod}\mbox{-}S$ to ${\rm Mod}\mbox{-}R$, can be seen from the commutative diagram below,

$\xymatrix{R \ar[r] \ar[d]_f & {\rm Ab}(R) \ar@{.>}[d]^{{\rm Ab}(f)} \\ S \ar[r] \ar[dr]_M & {\rm Ab}(S) \ar[d]^{{\rm ev}_M} \\ & {\bf Ab}}$

\noindent noting that the composition $ Mf $ is $ _RM:R\rightarrow {\bf Ab} $ and that, by the above lemma, the, exact, composition $ {\rm ev}_{_SM}{\rm Ab}(f) $ must be equivalent to $ {\rm ev}_{_RM}$.

If $ R\xrightarrow{f} S\xrightarrow{g} T $ in $ {\mathcal Rng} $ then $ {\rm Ab}(gf) $ is naturally equivalent to (can even, as in Section \ref{secabexfinacc}, be taken to equal) $ {\rm Ab}(g){\rm Ab}(f) $; so $ {\rm Ab}(-) $ is a type of 2-functor.

\vspace{4pt}

A possibly better embedding of algebraic geometry {\it per se} is the replacement of $ R $ by $ {\mathcal A}(R) $ (which, in the case that $R$ is right coherent, is just the category ${\rm mod}\mbox{-}R$).  In the case that $ R $ is commutative noetherian this is equivalent, {\it via} the the category of injective $R$-modules (on account of the natural bijection between primes and indecomposable injectives), to considering $ {\rm Spec}(R)$. It turns out, however, that $ R\mapsto  {\mathcal A}(R) $ is functorial only if we restrict to flat morphisms between rings. There is also the issue that this process seems not to capture closed subschemes - just the reduced variety (though this is not surprising since we are capturing only $ {\rm Spec}(R) $).

Using that $ {\rm Ex}({\rm Ab}({\mathcal R}),{\bf Ab})\simeq {\mathcal R}\mbox{-}{\rm Flat} $ (\ref{cohexgen}) we have the following.  Recall from \cite[\S 6]{PreRajShv} that ${\mathcal A}(R)=(R\mbox{-}{\rm mod}, {\bf Ab})/\{F: F(_RR)=0\}$.

\begin{prop}\label{Sfpflat} If $ f:R\rightarrow S $ is a homomorphism of rings then this induces, {\it via} $ {\rm Ab}(f)$ in the representation of ${\mathcal A}(R)$ as a quotient category of ${\rm Ab}(R)$, an exact functor $ {\mathcal A}(f):{\mathcal A}(R)\rightarrow {\mathcal A}(S) $ iff $ S\in \langle _R R\rangle  $ that is, iff $ _RS $ is fp-flat in the terminology of \cite{GarkGendual}.
In particular if $_RS$ is flat then $ {\mathcal A}(f) $ is an exact functor and, if $ R $ is right coherent, then this condition is also necessary.
\end{prop}
\begin{proof} For $ F\in {\rm Ab}(R)$, $ {\rm Ab}(f)F(_S S)=\overline{F}(_RS) $ which will be $ 0 $ for all such $ F $ iff $ _RS\in \langle _R R\rangle $. Thus $ _RS\in \langle _R R\rangle$  iff $ {\rm Ab}(f){\mathcal Z}_R\subseteq {\mathcal Z}_S$ (where ${\mathcal Z}_R$ denotes the Serre subcategory of functors which annihilate the module $R$).
\end{proof}

If $R$ is right coherent, so ${\mathcal A}(R)={\rm mod}\mbox{-}R$, then the functor ${\mathcal A}(f)$  is just  $M_R \mapsto M\otimes_RS_S$.  For, since $_RS$ is fp-flat and every $R$-module in ${\mathcal A}(R)$ is finitely presented, that functor is exact; also the two functors agree on $R_R$, taking it to $S_S$; therefore, by exactness, these functors agree on (projective presentations of) finitely presented modules.  Indeed, by \cite[3.1]{Garkrelhom} (and \ref{abfrng}) this description of ${\mathcal A}(f)$ is valid for any ring $R$.

In fact, and somewhat explaining this (in view of \ref{cohexgendual}) we have the following.

\begin{prop}\label{Sfpflat2} If $ f:R\rightarrow S $ is a homomorphism of rings then the restriction of scalars functor $ {\rm Mod}\mbox{-}S\rightarrow {\rm Mod}\mbox{-}R $ takes $ {\rm Inj}\mbox{-}S $ to $ {\rm Inj}\mbox{-}R $ iff $ _RS $ is flat.
\end{prop}
\begin{proof} If $ f $ is flat then for $ N_R $ and $ M_S$, $ {\rm Ext}^1_R(N_R,M_R) \simeq  {\rm Ext}^1_S(N\otimes _RS,M_S) $ so, if $ M_S $ is injective then so is $ M_R$.

For the converse we use that a module $ M_R $ is flat iff $ {\rm Hom}_{\mathbb Z}(M,{\mathbb Q}/{\mathbb Z}) $ is an injective $ R$-module (e.g. \cite[I.10.5]{Ste}). In particular $ S_R $ is flat iff $ {\rm Hom}_{\mathbb Z}(S,{\mathbb Q}/{\mathbb Z})_R $ is injective, so, if restriction of scalars preserves injectives then $ S_R $ is flat.
\end{proof}

That is, attempting to use $ {\mathcal A}(R) $ in place of $ {\rm Ab}(R) $ is no more than using $ {\rm Ab}(R) $ then trying to restrict to ${\rm Inj}\mbox{-}R$.

\begin{example}\label{nonredsch}\marginpar{nonredsch} Let $ R=k[\epsilon] = k[X]/\langle X\rangle ^2$
The definable category $ R\mbox{-}{\rm Proj} $ corresponding to $ {\mathcal A}(R)={\rm mod}\mbox{-}R $ has no proper definable subcategories. Yet $ R $ is the coordinate ring of a non-reduced scheme which has the proper subscheme corresponding to $ R\rightarrow R/\langle \epsilon\rangle  $ - a non-flat morphism.
\end{example}

This shows that the second embedding does not capture all the geometry we would wish; the first embedding does preserve all the structure but in general represents much more than the geometry of which $R$ is the coordinate ring.

\section{The 2-category of definable additive categories}\label{secdefcats}

A category is {\bf definable} if it is equivalent to a {\bf definable subcategory} of a module category $ {\rm Mod}\mbox{-}{\mathcal R} $, meaning a full subcategory which is closed under direct products, direct limits and pure subobjects.

Examples include, as well as module categories, the finitely accessible additive categories with products; in particular locally finitely presentable (in the terminology of \cite{AdRo}, \cite{GabUl}) additive categories are definable.  So both the category of torsionfree and the category of torsion ({\it sic}) abelian groups are definable.  Any definable subcategory of a definable category is definable.  The category of injective abelian groups is definable but has no non-zero finitely presented object.  The book \cite{PreNBK} is, in part, a compendium of examples of these.

Those are the objects of $ {\mathbb D}{\mathbb E}{\mathbb F}$; the morphisms are those functors which commute with direct products and direct limits (equivalently, see \cite[\S 25]{PreDefAddCat}, the model-theoretic interpretation functors).  These categories were originally studied in the context of the model theory of modules (see \cite{Zie}, \cite{PreBk}), as those subclasses of module categories which are axiomatised by implications between pp formulas; they were focussed on and named by Crawley-Boevey (\cite{C-BTrond}).  In fact they arise in a variety of contexts and, as seen in \ref{3cats}, they may be (re)presented in diverse ways.

Recall that ${\rm fun}({\mathcal D})$ denotes the skeletally small abelian category $({\mathcal D},{\bf Ab})^{\rightarrow \prod}$ and that in case ${\mathcal D}={\rm Mod}\mbox{-}{\mathcal R}$ this category, which we also write as ${\rm fun}\mbox{-}{\mathcal R}$, is equivalent to $({\rm mod}\mbox{-}{\mathcal R}, {\bf Ab})^{\rm fp}$.  In \cite[\S 12.3]{PreNBK}, ${\rm fun}({\mathcal D})$ was defined {\it via} localisation from the case ${\mathcal D} = {\rm Mod}\mbox{-}{\mathcal R}$ and here we are using the main theorem, 12.10, of \cite{PreDefAddCat} to define it as above directly from ${\mathcal D}$.  There is a natural correspondence, {\it via} annihilation, between definable subcategories of ${\mathcal D}$ and Serre subcategories of ${\rm fun}({\mathcal D})$.

\begin{theorem}\label{ppsortquot} (\cite{HerzCat},\cite{KraHab}; see, e.g., \cite[12.3.20]{PreNBK}) If $ {\mathcal D}' $ is a definable subcategory of the definable category $ {\mathcal D} $ and $ {\mathcal S}_{{\mathcal D}'} =\{ F\in {\rm fun}({\mathcal D}): FD=0 \mbox{ for all } D\in {\mathcal D}'\} $ denotes its annihilator in $ {\rm fun}({\mathcal D}) $ then ${\mathcal S}_{{\mathcal D}'}$ is a Serre subcategory of ${\rm fun}({\mathcal D})$ and the quotient category $ {\rm fun}({\mathcal D})/{\mathcal S}_{{\mathcal D}} $ is, in a natural way, naturally equivalent to $ {\rm fun}({\mathcal D}')$.  Every Serre subcategory of ${\rm fun}({\mathcal D})$ arises in this way.
\end{theorem}

We restate part of \ref{3cats} for convenience of reference.

\begin{theorem}\label{defisexfun} Suppose that ${\mathcal D}$ is a definable category; then ${\mathcal D} \simeq {\rm Ex}({\rm fun}({\cal D}),{\bf Ab})$.

Suppose that ${\mathcal A}$ is a skeletally small abelian category, then ${\mathcal A} \simeq {\rm fun}\big( {\rm Ex}({\mathcal A}, {\bf Ab})\big)$.
\end{theorem}

In one direction, the first equivalence takes an object $D$ of ${\mathcal D}$ to ${\rm ev}_D$, the evaluation of $F\in ({\mathcal D},{\bf Ab})^{\rightarrow \prod}$ at $D$.  The {\em description} of the other direction of that equivalence is less canonical because there are many ways in which the objects of ${\mathcal D}$ may be regarded; for instance, $R$-modules may be presented in the usual 1-sorted way but may alternatively be regarded as left exact functors on $({\rm mod}\mbox{-}R)^{\rm op}$ ({\it via} the restricted Yoneda embedding $ {\rm Mod}\mbox{-}R \rightarrow  (({\rm mod}\mbox{-}R)^{\rm op},{\bf Ab}) $ given by $M\mapsto (-,M)\upharpoonright {\rm mod}\mbox{-}R$) or as exact functors on ${\rm fun}\mbox{-}R$ (in model-theoretic terms this is making various choices of language for the structures in ${\mathcal D}$).  In any case, we may represent ${\mathcal D}$ as a definable subcategory of ${\rm Mod}\mbox{-}{\mathcal R}$ for some preadditive ${\mathcal R}$ (where ${\mathcal R}={\rm fun}({\mathcal D})$ would be the canonical choice) and then ${\rm fun}({\mathcal D})$ is a localisation of ${\rm fun}\mbox{-}{\mathcal R} = {\rm Ab}({\mathcal R}^{\rm op})$ so then, given $F$, an exact functor on ${\rm fun}({\mathcal D})$, we define the corresponding object $D$, as a functor from ${\mathcal R}^{\rm op}$ to ${\bf Ab}$, as taking an object $P$ in ${\mathcal R}$ to $F((-,(-,P))_{\mathcal D})$ where $(-,(-,P))_{\mathcal D}$ is the localisation of the representable functor $(-,(-,P))\in {\rm Ab}({\mathcal R}^{\rm op})$ at the Serre subcategory, ${\mathcal S}_{\mathcal D}$.  (The second equivalence is, in each direction, just evaluation.)

\vspace{4pt}

If $ {\mathcal D}={\rm Ex}({\mathcal A},{\bf Ab}) $ then the ({\bf elementary}) {\bf dual} category of $ {\mathcal D} $ is $ {\mathcal D}^{\rm d}={\rm Ex}({\mathcal A}^{\rm op},{\bf Ab}) $ and we have $ {\rm fun}({\mathcal D}^{\rm d})=({\rm fun}({\mathcal D}))^{\rm op}$ which we denote by ${\rm fun}^{\rm d}({\mathcal D})$. For instance the elementary dual of a module category $ {\rm Mod}\mbox{-}{\mathcal R} $ is just $ {\mathcal R}\mbox{-}{\rm Mod}$ and the duality between their functor categories, $({\mathcal R}\mbox{-}{\rm mod}, {\bf Ab})^{\rm fp} \simeq \big( ({\rm mod}\mbox{-}{\mathcal R}, {\bf Ab})^{\rm fp}\big) ^{\rm op}$, is in \cite{GrJeDim} and \cite{AusDual} (at least for ${\mathcal R}$ a ring).  The classical (retrospective) example, apart from the right and left module categories, is the duality, see e.g.~\cite[3.4.24]{PreNBK}, for $ R $ a right coherent ring, between $ {\rm Abs}\mbox{-}R $ (the absolutely pure right modules) and $ R\mbox{-}{\rm Flat} $ (the flat left modules).

\begin{prop}\label{dualindbij} (e.g.~\cite[12.4.1]{PreNBK}) If ${\mathcal D}$ is a definable category and ${\mathcal D}^{\rm d}$ its elementary dual then the duality ${\rm fun}({\mathcal D}) \simeq \big({\rm fun}({\mathcal D}^{\rm d}\big)^{\rm op}$ induces order-preserving bijections ${\rm Ser}({\rm fun}({\mathcal D}))\simeq {\rm Ser}({\rm fun}({\mathcal D}^{\rm d}))$ and ${\rm Sub}({\mathcal D}) \simeq {\rm Sub}({\mathcal D}^{\rm d})$ between the Serre subcategories of ${\rm fun}({\mathcal D})$ and ${\rm fun}({\mathcal D}^{\rm d})$ and between the definable subcategories of ${\mathcal D}$ and ${\mathcal D}^{\rm d}$.
\end{prop}

Let ${\mathcal D}$ be a definable category.  A monomorphism in ${\mathcal D}$ is {\bf pure} if some ultrapower of it is split.  The pure monomorphisms thus intrinsically defined are the restrictions to ${\mathcal D}$ of the pure monomorphisms in ${\rm Mod}\mbox{-}{\mathcal R}$ (these can be characterised in many ways) whenever ${\mathcal D}$ is embedded as a definable subcategory.  An object $D\in {\mathcal D}$ is {\bf pure-injective} if it is injective over the pure monomorphisms in ${\cal D}$.  Denote by ${\rm Pinj}({\mathcal D})$ the full subcategory of pure-injective objects of ${\mathcal D}$; it is cofinal in ${\mathcal D}$ in the sense that every object $D\in{\mathcal D}$ purely embeds in a pure-injective object, indeed $D$ has a unique-to-isomorphism-over-$D$ minimal such extension, termed its {\bf pure-injective hull} and denoted $H(D)$.

Let ${\rm pinj}({\mathcal D})$ denote the set of isomorphism classes of indecomposable (this excludes $0$) pure-injective objects of ${\mathcal D}$ (it is indeed a set).  The {\bf Ziegler topology} on ${\rm pinj}({\mathcal D})$ has, for a basis of open sets, the sets $(F)=\{ N\in {\rm pinj}({\mathcal D}): FN\neq 0\}$ as $F$ ranges over ${\rm fun}({\mathcal D})$.  The {\bf rep-Zariski topology} on ${\rm pinj}({\mathcal D})$ has for a basis of open sets the $[F]=(F)^{\rm c} =\{ N\in {\rm pinj}({\mathcal D}): FN= 0\}$.  These spaces are denoted ${\rm Zg}({\mathcal D})$ and ${\rm Zar}({\mathcal D})$ respectively.  In the case that ${\mathcal D} ={\rm Mod}\mbox{-}{\mathcal R}$, respectively ${\mathcal D}={\mathcal R}\mbox{-}{\rm Mod}$, we write ${\rm Zg}_{\mathcal R}$, respectively $_{\mathcal R}{\rm Zg}$, and similarly for ${\rm Zar}$.  There are many equivalent ways to define these topologies and much is known about them, for which I refer to \cite{PreNBK}.  And, of course, use of the name ``Zariski" indicates a generalisation of that spectrum (see \cite[p.~200ff.]{PreRem} or \cite[\S\S 14.1, 14.4]{PreNBK}).  Here is the basic connection with definable subcategories.

\begin{theorem}\label{defsubcatzg} (see e.g.~\cite[5.1.4, 12.4.1]{PreNBK}) Let ${\mathcal D}$ be a definable category.  The definable subcategories of ${\mathcal D}$ are in natural bijection with the closed subsets of its Ziegler spectrum  ${\rm Zg}({\mathcal D})$, indeed each definable subcategory is generated as such by the indecomposable pure-injectives in it.
\end{theorem}

The bijection takes ${\mathcal D}' \in {\rm Sub}({\mathcal D})$ to its intersection with ${\rm pinj}({\mathcal D})$ - its {\bf support} - and, in the other direction, takes a subset of ${\rm Zg}({\mathcal D})$ to the definable subcategory that it generates.

Also associated to ${\mathcal D}$ is the locally coherent Grothendieck category which can be obtained as follows:  if ${\mathcal D}$ is a definable subcategory of ${\rm Mod}\mbox{-}{\mathcal R}$ then we know that ${\rm fun}({\mathcal D}) =({\rm mod}\mbox{-}{\mathcal R}, {\bf Ab})^{\rm fp}/{\mathcal S}_{\mathcal D}$.  The Serre subcategory ${\mathcal S}_{\mathcal D}$ generates a hereditary torsion theory of finite type on the locally coherent Grothendieck category $({\rm mod}\mbox{-}{\mathcal R}, {\bf Ab})$ (see, e.g., \cite[\S 11.4]{PreNBK}) and the localisation, $({\rm mod}\mbox{-}{\mathcal R}, {\bf Ab})_{\mathcal D}$, of $({\rm mod}\mbox{-}{\mathcal R}, {\bf Ab})$ at this torsion theory is a locally coherent Grothendieck category - we denote it ${\rm Fun}({\mathcal D})$ - and we have $({\rm Fun}({\mathcal D}))^{\rm fp}={\rm fun}({\mathcal D})$.  These categories are considered in Section \ref{sectopos}.

\begin{rmk}\label{Fundirect} The category ${\rm Fun}({\mathcal D})$ can also be obtained directly as $({\mathcal D}, {\bf Ab})^{\rightarrow}$.  To see this we consider first the case that ${\mathcal D} ={\rm Mod}\mbox{-}{\mathcal R}$ and hence that ${\rm Fun}({\mathcal D}) = ({\rm mod}\mbox{-}{\mathcal R}, {\bf Ab})$; then the equivalence of this category with those functors from ${\rm Mod}\mbox{-}{\mathcal R}$ to ${\bf Ab}$ which commute with direct limits, and hence which are determined by their actions on ${\rm mod}\mbox{-}{\mathcal R}$, is easily seen (and essentially in \cite{AusLg}).  In the general case, ${\mathcal D}$ is a definable subcategory of some module category ${\rm Mod}\mbox{-}{\mathcal R}$ and ${\rm Fun}({\mathcal D})$ is the localisation of ${\rm Fun}\mbox{-}{\mathcal R}$ described above.  But the action of the localisation of $F\in {\rm Fun}\mbox{-}{\mathcal R}$ is just restriction of its action from ${\rm Mod}\mbox{-}{\mathcal R}$ to ${\mathcal D}$.  Since direct limits in ${\mathcal D}$ agree with those in ${\rm Mod}\mbox{-}{\mathcal R}$, it follows that any $F\in {\rm Fun}({\mathcal D})$ does commute with direct limits.  For the converse, if $F\in ({\mathcal D}, {\bf Ab})^{\rightarrow}$ then it will be enough to show that $F$ has an extension to a functor on ${\rm Mod}\mbox{-}{\mathcal R}$ which commutes with direct limits (since then $F$ will be the localisation of this extension and hence in ${\rm Fun}({\mathcal D})$).  But by, e.g.~\cite[2.4]{CrivPreTorr}, any definable subcategory of ${\rm Mod}\mbox{-}{\mathcal R}$ is contravariantly finite=precovering in ${\rm Mod}\mbox{-}{\mathcal R}$ (as well as covariantly finite - see, e.g.~\cite[3.4.42]{PreNBK}) and so its left Kan extension exists and will, being left adjoint to restriction, commute with direct limits, as required.
\end{rmk}

\subsection{An intrinsic definition of definable category?}\label{secintrin}

We have two, equivalent, definitions of the notion of definable category:  as a definable subcategory of a category of ${\mathcal R}$-modules; as the category of exact functors on a small abelian category ${\mathcal A}$.  Both are, however, definitions in terms of some representation, and although we can give some sort of ``intrinsic" definition of ``definable category" (see below) it would be desirable to have a list of (preferably easily-checkable) category-theoretic properties which cuts out exactly the definable categories.

Apart from being additive, we would require having direct products and direct limits.  Since ultraproducts are certain direct limits of direct products these conditions are enough to give a category an internal theory of purity since we can define a pure monomorphism to be one, some ultraproduct of which is split, and similarly we may define pure epimorphisms.  So we might also add the conditions that these have, respectively, cokernels and kernels.  We should also add some ``smallness" condition since the properties mentioned so far are shared by all Grothendieck, indeed AB5, abelian categories and not every such category is definable.

\begin{theorem}\label{defgroth} A Grothendieck abelian category is definable iff it is locally finitely presented.
\end{theorem}
\begin{proof} If ${\mathcal G}$ is Grothendieck abelian and definable  set ${\mathcal A} ={\rm fun}({\mathcal G})$.  Then ${\mathcal G}$ is the definable subcategory, ${\rm Ex}({\mathcal A},{\bf Ab})$ of ${\mathcal A}\mbox{-}{\rm Mod}$ consisting of the exact functors.  We claim that the inclusion, $i$, of ${\mathcal G}$ in ${\mathcal A}\mbox{-}{\rm Mod}$ has a left adjoint, that is, that ${\mathcal G}$ is a localisation of ${\mathcal A}\mbox{-}{\rm Mod}$.

First note that $i$ preserves kernels since if $0\rightarrow K \rightarrow F \rightarrow G$ is exact with $F,G \in {\mathcal G}$ then, since those are exact, so is $K$.  Since ${\mathcal G}$ is definable, direct products in ${\mathcal G}$ coincide with those in ${\mathcal A}\mbox{-}{\rm Mod}$.  Thus $i$ preserves limits.  We also have the solution set condition of the Adjoint Functor Theorem because every definable subcategory is covariantly finite and hence, given any $M\in {\mathcal A}\mbox{-}{\rm Mod}$ there is an arrow $M\rightarrow G\in {\mathcal G}$ through which every morphism from $M$ to an object of ${\mathcal G}$ factors.

Therefore $i$ has a left adjoint $Q$ which is left exact and hence $Q$ is a localisation at a hereditary torsion theory.  Since ${\mathcal G}$ is definable the inclusion $i$ commutes with directed colimits so (see \cite[11.1.23]{PreNBK}) this torsion theory is of finite type and hence, see \cite[11.1.27]{PreNBK}, the localised category, ${\mathcal G}$, is locally finitely presented, as required.
\end{proof}

Certainly any definable category ${\mathcal D}$ is accessible:  if we set  $\kappa$ to be $|{\mathcal R}|+\aleph_0$, if ${\mathcal D}$ is a definable subcategory of ${\rm Mod}\mbox{-}{\mathcal R}$, or the number of morphisms in a skeleton of ${\rm fun}({\mathcal D})$ then every object of ${\mathcal D}$ is a structure for a functional language of cardinality $\kappa$ and so (for instance by the downwards Lowenheim-Skolem Theorem) is the direct limit of its subobjects (or pure, or even elementary, subobjects) of cardinality $\leq \kappa$ (or $\kappa^+$ if $\kappa$ is not regular). Furthermore, each object of cardinality $\leq \kappa$ is $ \kappa^+$-presentable, and there is just a set of these up to isomorphism.  Thus ${\mathcal D}$ is $\kappa^+$-accessible (though, as already mentioned, not necessarily finitely accessible).  If we were content to work with infinitary languages then this would be enough (see \cite[5.35]{AdRo}, \cite[9.1.7]{Hod}, also \cite{Hu}).

For an intrinsic definition we suppose just that ${\mathcal D}$ is additive with products and direct limits.  The proof of \ref{funtoCab} below shows that $({\mathcal D}, {\bf Ab})^{\rightarrow \prod}$ is an abelian category.  If we also suppose that ${\mathcal D}$ has a $\varinjlim$-generating set of objects then $({\mathcal D}, {\bf Ab})^{\rightarrow \prod}$ is also skeletally small.  Of course ${\mathcal D}$ then embeds naturally, {\it via} $D\mapsto {\rm ev}_D$, into ${\rm Ex}\big(({\mathcal D}, {\bf Ab})^{\rightarrow \prod}, {\bf Ab}\big)$ and we know, using \ref{defisexfun}, that this will be an equivalence iff ${\mathcal D}$ is definable.  In some sense that is an intrinsic characterisation but it is considerably less satisfactory than would be a list of conditions which could be checked directly, because it is not clear how one might in general check that the embedding is an equivalence.

\subsection{Extending from ${\rm Pinj}({\mathcal D})$ to ${\mathcal D}$}.

If ${\mathcal D}$ is a definable category then any morphism $F:{\mathcal D} \rightarrow {\mathcal C}$ in $ {\mathbb D}{\mathbb E}{\mathbb F} $ restricts to a functor on the full subcategory, $ {\rm Pinj}({\mathcal D}) $, of pure-injective objects of ${\mathcal D}$ to $ {\mathcal C} $ (indeed, see \ref{interpprespi}, to ${\rm Pinj}({\mathcal C})$).  That restricted functor commutes with direct products and with those direct limits of pure-injectives where the direct limit object happens to be in ${\rm Pinj}({\mathcal D})$ (for short we may describe that second condition as ``commuting with those direct limits in $ {\rm Pinj}({\mathcal D})$'').  In this section we consider the converse.  That is:  suppose ${\rm Pinj}({\mathcal D}) \rightarrow {\mathcal C}$ commutes with direct products and with those direct limits diagrams of objects in $ {\rm Pinj}({\mathcal D}) $ whose direct limit also is in $ {\rm Pinj}({\mathcal D}) $; then does this extend to a morphism in $ {\mathbb D}{\mathbb E}{\mathbb F} $ from ${\mathcal D}$ to ${\mathcal C}$?  We show that this is so.  We also consider the question of whether natural transformations between morphisms in $ {\mathbb D}{\mathbb E}{\mathbb F} $ are determined by their restrictions to the pure-injective objects.  Again the answer is positive.

These issues already arise in the proof (\cite[12.10]{PreDefAddCat}) of the fact that, for $ {\mathcal D}\in {\mathbb D}{\mathbb E}{\mathbb F}$, the category $ {\rm fun}({\mathcal D}) $ is equivalent to $ ({\mathcal D},{\bf Ab})^{\rightarrow \prod } $.  Let us outline the shape of that proof since we will be reconsidering parts of it here.  The first part consists of showing that $ ({\rm Pinj}({\mathcal D}),{\bf Ab})^{\prod} \simeq ({\rm Pinj}({\mathcal D}),{\bf Ab})^{\rm fp} $ and then applying the fact from \cite{KraEx} (see \cite[12.2]{PreDefAddCat}) that $ ({\rm Pinj}({\mathcal D}),{\bf Ab})^{\rm fp}\simeq  ({\rm Fun}^{\rm d}({\mathcal D}))^{\rm op} $ in order to identify $ ({\rm Pinj}({\mathcal D}),{\bf Ab})^{\prod} $ with the opposite of the dual ``large'' functor category $ {\rm Fun}^{\rm d}({\mathcal D})$. Each of these latter two categories has a natural action on $ {\rm Pinj}({\mathcal D}) $ and, in the proof of \cite[12.10]{PreDefAddCat}, it is shown that the identification respects this. Then we take a functor $G \in ({\mathcal D},{\bf Ab})^{\rightarrow \prod }$.  The action of that functor on $ {\rm Pinj}({\mathcal D})$ is then shown to coincide with the action on ${\rm Pinj}({\mathcal D})$ of a functor of the form $ F_{\phi/\psi} $ for some pp-pair $ \phi/\psi$.  The proof goes on to show that since these functors agree on $ {\rm Pinj}({\mathcal D}) $ they agree on $ {\mathcal D}$.  We will generalise that last part here, replacing the codomain $ {\bf Ab} $ by an arbitrary definable category and replacing the natural isomorphism between the two functors (the restrictions of $G$ and $F$ to ${\rm Pinj}({\mathcal D})$) by any natural transformation.  And the first part of that proof, which we will re-do rather more cleanly than in \cite{PreDefAddCat} (where category-theoretic and element-based argumentation sit uncomfortably together), will allow us to answer the first question.

We will make use of the fact that every definable category ${\mathcal D}$ has an {\bf elementary cogenerator}, that is  $ N\in {\rm Pinj}({\mathcal D}) $ which is such that every object of $ {\mathcal D} $ is a pure subobject of a direct product of copies of $ N$ (see \cite[9.36]{PreBk} or \cite[5.3.52]{PreNBK}). We write $\langle - \rangle$ for the definable subcategory generated by $(-)$.  First we recall the following.

\begin{prop}\label{interpprespi} (see, e.g., \cite[13.1]{PreDefAddCat}) Suppose that $I:{\mathcal C} \rightarrow {\mathcal D}$ is a morphism in ${\mathbb D}{\mathbb E}{\mathbb F}$. Then $I$ preserves pure embeddings and pure-injectivity. If ${\mathcal D}'\subseteq {\mathcal D}$ is a definable subcategory of ${\mathcal D}$ then $I^{-1}{\mathcal D}' =\{ C\in {\mathcal C}: IC\in {\mathcal D}'\}$ is a definable subcategory of ${\mathcal C}$.
\end{prop}

\begin{prop}\label{defsubgen} If $ I:{\mathcal D}\rightarrow {\mathcal C} $ is a morphism in ${\mathbb D}{\mathbb E}{\mathbb F}$ then every object of $ \langle I{\mathcal D}\rangle  $ is a pure subobject of an object of the form $ ID$. If $ N $ is an elementary cogenerator for $ {\mathcal D} $ then $ IN $ is an elementary cogenerator for $ \langle I{\mathcal D}\rangle $.
\end{prop}
\begin{proof} The first statement will follow from the second by \ref{interpprespi} and since $ I $ commutes with direct products. Since $ N $ is an elementary cogenerator it is clear that $ \langle IN\rangle =\langle I{\mathcal D}\rangle  $ and hence the support of $ IN $ in the Ziegler spectrum of $ {\mathcal C} $ is $ {\rm Zg}(\langle I{\mathcal D}\rangle )$. Therefore every $ N_1\in {\rm Zg}(\langle I{\mathcal D}\rangle ) $ is a direct summand of an ultrapower of (a power of) $ IN $.  Since every ultrapower of $ N $ is pure in a direct power of $ N$, so the same is true of $ IN$, $ N_1 $ is, therefore, a direct summand of a power of $ IN$. Then the fact that every point of ${\rm Zg}(\langle I{\mathcal D}\rangle )$ is a direct summand of a power of $IN$ is enough (see \cite[5.3.50]{PreNBK}) to imply that $ IN $ is an elementary cogenerator of $\langle I{\mathcal D}\rangle$.
\end{proof}

\begin{theorem}\label{extnattran} Suppose that $F, G:{\mathcal D} \rightrightarrows {\mathcal C}$ are functors in ${\mathbb D}{\mathbb E}{\mathbb F}$ and that $\tau':F\upharpoonright {\rm Pinj}({\mathcal D}) \rightarrow G \upharpoonright {\rm Pinj}({\mathcal D})$ is a natural transformation between their restrictions to the pure-injectives of ${\mathcal D}$.  Then there is a, unique, natural transformation $\tau:F\rightarrow G$ which restricts to $\tau'$.  If $\tau'$ is a natural isomorphism then so is $\tau$.
\end{theorem}
\begin{proof} We use an argument from the proof of \cite[12.10]{PreDefAddCat}.  Let $M\in {\mathcal D}$ and choose a pure embedding $M\xrightarrow{i} N\in {\rm Pinj}({\mathcal D})$ into a pure-injective.  Choose a set $I$ and an ultrafilter ${\mathcal F}$ on $I$ such that both $M^I/{\mathcal F}$ and $N^I/{\mathcal F}$ (and hence also their images under $F$ and $G$) are pure-injective (see, e.g., \cite[[4.2.19]{PreNBK} or \cite[21.3]{PreDefAddCat}).  Consider the diagram shown, where $\Delta_{(-)}$ is the diagonal map into the ultraproduct.

$\xymatrix{ & FN^I/{\cal F} \ar[r]^{\tau'_{(N^I/{\mathcal F})}} & GN^I/{\cal F} \\
FN \ar[ur]^{\Delta_{FN}} \ar[rrr]^{\tau'_N} & & & GN \ar[ul]_{\Delta_{GN}} \\
& FM^I/{\cal F} \ar@{->}[uu]^<<<<{Fi^I/{\mathcal F}} \ar[r]^{\tau'_{(M^I/{\mathcal F})}} & GM^I/{\cal F} \ar@{->}[uu]_<<<<{Gi^I/{\mathcal F}} \\
FM \ar[uu]^{Fi} \ar[ur]_{\Delta_{FM}} & & & GM \ar[uu]_{Gi} \ar[ul]^{\Delta_{GM}} }$

\noindent We note that $F(N^I/{\cal F})= (FN)^I/{\cal F}$, $F\Delta_N =\Delta_{FN}$, $F(i^I/{\mathcal F}) = (Fi)^I/{\mathcal F}$ etc.  Also, \ref{interpprespi}, pure embeddings are taken to (pure) embeddings by $F$ and $G$.  The top square commutes since $\tau$ is a natural transformation, the back for the same reason and, using that $M^I/{\mathcal F}$ is pure-injective and what has just been noted, we see that the sides commute since $F$ and $G$ are functors.

From the construction we have that ${\rm im}(\Delta_{GN}) \,\cap\, \big((GM)^I/{\cal G} \big) = {\rm im}(\Delta_{GM}) \simeq GM$; that is, the right-hand side is a pullback.  So, working round the commutative squares, we obtain a unique map, which we denote $\tau_M$, from $FM$ to $GM$ making the whole diagram commute.  (Note that this is independent of choice of $N$ since, given another choice of initial embedding $i':M\rightarrow N'$, we can use $(i,i'):M\rightarrow N\oplus N'$, to define ``$\tau_M$'' which, one may check, restricts to the otherwise-constructed $\tau_M$s.

Arguing similarly one checks that the $\tau_M$ cohere to form a natural transformation from $F$ to $G$.  The last statement follows easily.
\end{proof}

It is clear from the proof that it is sufficient that the natural transformation $\tau$ be defined on some cofinal class of pure-injectives, indeed being defined on an elementary cogenerator would be enough.

\begin{prop}\label{commprodfp0} Suppose that $ {\mathcal C}, {\mathcal D} $ are additive categories with products and coproducts and suppose that ${\mathcal C}$ is abelian.  Let $ ({\mathcal D},{\mathcal C})^{\prod} $ denote the category of those functors from $ {\mathcal D} $ to $ {\mathcal C} $ which commute with direct products. Suppose that $G\in ({\mathcal D},{\mathcal C})$ is generated.  Then $G\in ({\mathcal D},{\mathcal C})^{\prod}$ iff $G \in ({\mathcal D},{\mathcal C})^{\rm fp}$, the category of finitely presented functors from ${\mathcal D}$ to ${\mathcal C}$.
\end{prop}
\begin{proof} We recall that a functor $G$ from ${\mathcal D}$ to ${\mathcal C}$ is {\bf finitely generated} if it is a quotient of a (finite direct sum of) representable functor(s) and $G$ is {\bf finitely presented} if the kernel of such a presentation is itself finitely generated.  Beware that, because ${\mathcal D}$ has a proper class of objects, not every functor on it will be {\bf generated}, that is, determined by its action on a set of objects; that is, not every functor has a presentation (let alone a finite one).

For one direction, suppose that $ G\in ({\mathcal D},{\mathcal C})^{\rm fp}$, so there is a morphism $ f:N\rightarrow N' $ in $ {\mathcal D} $ such that $ (N',-)\xrightarrow{(f,-)} (N,-)\rightarrow G\rightarrow 0 $ is exact.  Since both $ (N',-) $ and $ (N,-) $ commute with products (by definition of direct product) and since, as an additive functor category, direct products in $ {\mathcal D},{\mathcal C}) $ are exact, it follows that $ G $ commutes with products.  Thus $({\mathcal D},{\mathcal C})^{\rm fp}$ is a subcategory of $({\mathcal D},{\mathcal C})^{\prod}$.

For the converse, suppose that $ G\in ({\mathcal D},{\mathcal C}) $ commutes with direct products, meaning that for any indexed set $ (N_i)_{i\in I} $ of objects of $ {\mathcal D} $ the canonical map $ \prod G\pi _i:G \prod _iN_i\rightarrow \prod _iGN_i$, where $ \pi _j: \prod N_i\rightarrow N_j $ are the canonical projections, is an isomorphism. Then, in the commutative diagram shown

$\xymatrix{(\bigoplus _j(N_j,-),G) \ar@{<.>}[d]_\simeq & & ((\prod _iN_i,-),G) \simeq G\prod N_i \ar[ll]_{(\bigoplus _j(\pi _j,-),G)} \ar@{.>}[dll]^{\prod G\pi _i}_\simeq \\ \prod _iGN_i \simeq \prod_i ((N_i,-),G)}$

\noindent it follows that $ (\bigoplus _j(\pi _j,-),G)$ is an isomorphism.

Since $ (\bigoplus _i(\pi _i,-),G) $ is restriction of morphisms along $ \bigoplus _i(\pi _i,-): \bigoplus _j(N_j,-) \rightarrow (\prod _iN_i,-) $ it follows that each morphism $ p:\bigoplus _i(N_i,-)\rightarrow G $ has an extension to a morphism $ p':(\prod _iN_i,-)\rightarrow G$. So, if there is an epimorphism from $ \bigoplus _i(N_i,-) $ to $ G $ then there is one from $ (\prod _iN_i,-)$. That is, if $ G $ is generated then $ G $ is finitely generated.
\end{proof}

\begin{cor}\label{commprodfp} Suppose that $ {\mathcal C}, {\mathcal D} $ are definable categories with ${\mathcal C}$ abelian.  Then $ ({\rm Pinj}({\mathcal D}),{\mathcal C})^{\prod} =({\rm Pinj}({\mathcal D}),{\mathcal C})^{\rm fp}$.
\end{cor}
\begin{proof}  By \ref{commprodfp0} it must be shown that every functor $G$ which commutes with products is generated.  First we note, as a general point, that given $ A,B $ in an additive category with coproducts, every morphism from a coproduct $ A^{(I)} $ of copies of $ A $ to $ B $ factors through the canonical morphism $ c:A^{(A,B)}\rightarrow B $ where $ c=(h)_{h\in(A,B)}$. For, if $ f:A^{(I)}\rightarrow B $ is $ f=(f_i)_i $ with $ f_i:A_{(i)}\rightarrow B $ then, for each $ i\in I $, choose an isomorphism $ A_{(i)}\rightarrow A $ (the copy appearing in ``$A^{(A,B)}$", which we assume to be replicated in a specified way at each $ h\in(A,B)$) and define $ h(i) $ to be the composition shown.

$\xymatrix{A_{(i)} \ar[dr]^{f_i} \ar[d] \\ A \ar[r]_{h(i)} & B}$

\noindent Define $ g_i $ to be the composition $ A_{(i)}\rightarrow A\xrightarrow{i_{h(i)}} A^{(A,B)}$, yielding $ g=(g_i)_i:A^{(I)}\rightarrow  A^{(A,B)}$. Then $ cg=f$: it is sufficient to check at each $ i\in I$, where it holds by definition of $ h(i)$, as required.

Now let $ N $ be an elementary cogenerator for $ {\mathcal D}$ and let $G\in ({\rm Pinj}({\mathcal D}),{\mathcal C})^{\prod} $.  So, if $ N'\in {\rm Pinj}({\mathcal D}) $ then there is a split exact sequence of the form $ 0\rightarrow N'\rightarrow N^I \rightarrow N''\rightarrow 0 $ and hence a split exact sequence $ 0\rightarrow (N'',-)\rightarrow (N^I ,-)\rightarrow (N',-)\rightarrow 0 $ in $ ({\rm Pinj}({\mathcal D}),{\mathcal C})$.  In particular any morphism $d':(N',-) \rightarrow G$ lifts to some $d:(N^I,-) \rightarrow G$.  Under the identification $((N^I,-),G) \simeq G(N^I) \simeq (GN)^I$, $d$ corresponds to some $\overline{d}=(d_i)_{i\in I}$ with $d_i\in GN \simeq ((N,-),G)$.  Thus we have an induced map $i\mapsto d_i:I\rightarrow GN$ and hence an induced $\chi_d:N^{GN}\rightarrow N^I$.  That, in turn, induces $(\chi_d,-):(N^I,-) \rightarrow (N^{GN},-)$ and so $(\chi_d,G):((N^{GN},-),G) \rightarrow ((N^I,-),G)$ which takes $c_1\in ((N^{GN},-),G)$ to $d$ where $c_1$ corresponds to $c\in ((N,-)^{(GN)},G)$ (notation as above) {\it via} the isomorphism $((N,-)^{(GN)},G) \simeq ((N,-),G)^{GN} \simeq ((N^{GN},-),G)$ (the latter being that $G$ commutes with products).  In particular $d$, and hence $d'$, factors through $c_1$ and we deduce that $c_1:(N^{GN},-) \rightarrow G$ is an epimorphism.
\end{proof}

Note that an abelian definable category is Grothendieck (hence, \ref{defgroth}, locally finitely presented): it is complete and well-powered, so it is sufficient to check that it has exact direct limits but, being a subcategory of a Grothendieck (functor) category and being closed in that category under direct limits, this follows.

\begin{lemma}\label{funtoCab} Suppose that ${\mathcal D}$ is a definable category and that ${\mathcal C}$ is an abelian definable category.  Then $({\mathcal D}, {\mathcal C})^{\rightarrow \prod}$, that is, ${\mathbb D}{\mathbb E}{\mathbb F}({\mathcal D}, {\mathcal C})$ is a skeletally small abelian category.
\end{lemma}
\begin{proof} It is easily checked that this category (or the equivalent, by \ref{3cats}, category ${\rm Ex}({\rm fun}({\mathcal C}), {\rm fun}({\mathcal D})$) is skeletally small (see the proof of \ref{funpitoCab} below) and additive.  If $\tau:F\rightarrow G$ is an arrow in $({\mathcal D}, {\mathcal C})^{\rightarrow \prod}$ then we can define ${\rm ker}(\tau)$ to be the functor taking $D\in {\mathcal D}$ to ${\rm ker}(\tau_D)$ (which exists since ${\mathcal C}$ is abelian) and having action on arrows given in the obvious way (see the diagram).

$\xymatrix{D \ar[d]_f & & 0 \ar[r] & {\rm ker}(\tau)D \ar[r] \ar@{.>}[d] & FD \ar[d]_{Ff} \ar[r]^{\tau_D} & GD \ar[d]^{Gf} \\ D' & & 0 \ar[r] & {\rm ker}(\tau)D' \ar[r] & FD' \ar[r]_{\tau_{D'}} & GD'}$

\noindent  Since in ${\mathcal C}$ both direct limits and direct products are exact, this functor ${\rm ker}(\tau)$ also commutes with these (as in the proof below) and the inclusion of it into $F$ is, indeed, the kernel of $\tau$.  Dually we obtain ${\rm coker}(\tau)$ also in $({\mathcal D}, {\mathcal C})^{\rightarrow \prod}$.  In a similar way we can see that every monomorphism in $({\mathcal D}, {\mathcal C})^{\rightarrow \prod}$ has all its components monomorphisms and is a cokernel, and dually for epimorphisms.
\end{proof}

\begin{lemma}\label{funpitoCab} Suppose that ${\mathcal D}$ is a definable category and that ${\mathcal C}$ is an abelian definable category.  Then $({\rm Pinj}({\mathcal D}), {\mathcal C})^{\rightarrow \prod}$ is a skeletally small abelian category.
\end{lemma}
\begin{proof} To see that this functor category is abelian we use that $({\rm Pinj}({\mathcal D}), {\mathcal C})^{\rightarrow \prod}$ is a full subcategory of the (large) functor category $({\rm Pinj}({\mathcal D}), {\mathcal C})$ in which both $\prod$ and $\varinjlim$ are exact, from which it follows easily (see the proof above) that $({\rm Pinj}({\mathcal D}), {\mathcal C})^{\rightarrow \prod}$ is abelian.  For example, if $\tau:F \rightarrow G$ is a morphism in $({\rm Pinj}({\mathcal D}), {\mathcal C})^{\rightarrow \prod}$ then we have, in $({\rm Pinj}({\mathcal D}), {\mathcal C})$, the exact sequence $\sigma: 0\rightarrow {\rm ker}(\tau) \rightarrow F \xrightarrow{\tau} G\rightarrow {\rm coker}(\tau) \rightarrow 0$.  So then, if $D=\varinjlim_\lambda D_\lambda$ is a direct limit in ${\rm Pinj}({\mathcal D})$, then we have the exact sequences $\sigma(D_\lambda)$ forming a directed system of exact sequences in ${\mathcal C}$ with exact direct limit.  But also $\sigma(D)$ is exact and so, comparing these sequences, we deduce that $\varinjlim_\lambda ({\rm ker}(\tau)\cdot (D_\lambda)) ={\rm ker}(\tau) \cdot D$.

To show that the category is skeletally small we use that the proof for the case ${\mathcal C} ={\bf Ab}$ is done within the proof of \cite[12.10]{PreDefAddCat}.  That proof works in the more general case but carrying that through would require setting up quite a bit of the background material (in particular that relating to the category ${\rm Fun}^{\rm d}({\mathcal D})$ which makes an appearance).  An easier alternative is to note that the case ${\mathcal C} ={\bf Ab}$ is enough since we may regard objects of ${\mathcal C}$ as modules over some small preadditive category ${\mathcal R}$, hence as multi-sorted structures with a sort for each object $P$ of ${\mathcal R}$.  Then the composition of a functor preserving products and directed colimits to ${\mathcal C}$ with evalution ($C \mapsto ((P,-),C)$) at a particular sort is a functor (preserving products and directed colimits).  There is just a set of sorts and a set of morphisms between them so, putting together the data from the separate sorts, and since we know that $({\rm Pinj}({\mathcal D}), {\bf Ab})^{\rightarrow \prod}$ is a set, we deduce that $({\rm Pinj}({\mathcal D}), {\mathcal C})^{\rightarrow \prod}$ is a set.
\end{proof}

We will now specialise to the case ${\mathcal C}= {\bf Ab}$; we could continue with the general case but we wish to quote, in the proof of \ref{complissub}, a result which is proved in the case ${\mathcal C}={\bf Ab})$.  In this case also, the proof of the quoted result would generalise easily enough but, rather than do that, we will say how to obtain the general case from what we do.

Let us set ${\mathcal B} = ({\rm Pinj}({\mathcal D}), {\bf Ab})^{\rightarrow \prod}$ and also suppose that ${\mathcal A}$ is a small abelian category such that ${\mathcal D} = {\rm Ex}({\mathcal A}, {\bf Ab})$.  Certainly the action of each object of ${\mathcal A}$ on ${\rm Pinj}({\mathcal D})$ commutes with direct products and those direct limits in ${\rm Pinj}({\mathcal D})$, so we have a functor from ${\mathcal A}$ to ${\mathcal B}$.

\begin{lemma}\label{emblang}  With notation as above, the functor from ${\mathcal A}$ to ${\mathcal B}$ is a faithful, full and exact embedding.
\end{lemma}
\begin{proof} If two objects of ${\mathcal A}$ agree on ${\rm Pinj}({\mathcal D})$ then, since every object of ${\mathcal D}$ is pure in a pure-injective object, they agree on all of ${\mathcal D}$; ditto, by \ref{extnattran}, for morphisms between such.  A sequence of morphisms in ${\mathcal A}$ is exact iff it is exact at each object of ${\mathcal D}$ iff  it is exact on each object of ${\rm Pinj}({\mathcal D})$, in other words, iff its image in ${\mathcal B}$ is exact; the fact that a sequence of functors/objects of ${\mathcal A}$ which is is exact on ${\rm Pinj}({\mathcal D})$ is exact on ${\mathcal D}$ follows again because every object of ${\mathcal D}$ is pure in an object of ${\rm Pinj}({\mathcal D})$.

By \ref{extnattran} the embedding of ${\mathcal A}$ into ${\mathcal B}$ is full.

Note, for later use, that this proof works with any definable abelian category ${\mathcal C}$ in place of ${\bf Ab}$.
\end{proof}

We are going to prove that the embedding of ${\mathcal A}$ into ${\mathcal B}$ is an equivalence.  In model-theoretic terms, regarding objects of ${\rm Pinj}({\mathcal D})$ as functors on ${\mathcal A}$ is regarding them as ${\mathcal A}$-structures (structures for the language of ${\mathcal A}$-modules) but, as we have just seen, and this is said in more detail below, they are also ${\mathcal B}$-structures.  So our question about extending functors from ${\rm Pinj}({\mathcal D})$ to ${\mathcal D}$ is equivalent to asking whether every ${\mathcal A}$-structure in the definable category ${\mathcal D}$ is also a ${\mathcal B}$-structure.

\begin{lemma}\label{embpinj}\marginpar{embpinj} Suppose that ${\mathcal D}$ is a definable category.  Let ${\mathcal B} = ({\rm Pinj}({\mathcal D}), {\bf Ab})^{\rightarrow \prod}$ - a skeletally small abelian category.  Then there is a natural full embedding of ${\rm Pinj}({\mathcal D})$ into the definable category $\overline{\mathcal D} = {\rm Ex}({\mathcal B}, {\bf Ab})$.  Each object $N\in {\rm Pinj}({\mathcal D})$ is pure-injective as a ${\mathcal B}$-structure.
\end{lemma}
\begin{proof} The embedding is that which takes $N\in {\rm Pinj}({\mathcal D})$ to the functor, ${\rm ev}_N$, evaluation-at-$N$.  That it is full can be argued as follows.  As above, we suppose that ${\mathcal D}={\rm Ex}({\mathcal A},{\bf Ab})$ for some small abelian category ${\mathcal A}$ and we have the restriction of actions of objects of ${\mathcal A}$ from ${\mathcal D}$ to ${\rm Pinj}({\mathcal A})$, giving the embedding from \ref{emblang} of ${\mathcal A}$ into ${\mathcal B}$.  Then any natural transformation, $\tau: {\rm ev}_N \rightarrow {\rm ev}_{N'}$, between evaluation-on-${\mathcal B}$ functors (i.e.~objects of ${\rm Pinj}({\mathcal D})$ regarded as ${\mathcal B}$-structures) restricts to one between evaluation-on-${\mathcal A}$ functors (i.e.~${\mathcal A}$-structures) and that is just a morphism from $N$ to $N'$ as objects of ${\mathcal D}$.

Now we show pure-injectivity as ${\mathcal B}$-structures.  Let us write $N_{\mathcal B}$ instead of ${\rm ev}_N$ to emphasise the view of these as structures.  Then, given $N\in {\rm Pinj}({\mathcal D})$, there is some index set $I$ and ultrafilter ${\mathcal F}$ on $I$ such that the ultrapower $N_{\mathcal B}^I/{\mathcal F}$ is pure-injective as a ${\mathcal B}$-structure.   It is, in particular, pure-injective as an ${\mathcal A}$-structure.  Consider the diagonal map but regarded as a (pure) embedding of ${\mathcal A}$-structures: $(N_{\mathcal B})_{\mathcal A} \rightarrow (N_{\mathcal B}^I/{\mathcal F})_{\mathcal A}$.  The first object can be identified with $N$ and hence this map is split.  Therefore the image of this map regarded in the category of ${\mathcal B}$-structures is split.  But, since the action of the objects of ${\mathcal B}$ commute with ultraproducts of objects of ${\rm Pinj}({\mathcal D})$ (since ultraproducts are certain direct limits of products), we have $N_{\mathcal B}^I/{\mathcal F} =( N^I/{\mathcal F})_{\mathcal B}$ and so the diagonal embedding of $N_{\mathcal B}$ into $N_{\mathcal B}^I/{\mathcal F}$ is split.  Thus $N_{\mathcal B}$ is pure-injective.
\end{proof}

Note that $\overline{\mathcal D}$ is naturally a subcategory of ${\mathcal D}$ {\it via} ${\mathcal A} \rightarrow {\mathcal B}$.

\begin{cor}\label{Dbar}\marginpar{Dbar} With notation as above, ${\rm Pinj}(\overline{\mathcal D}) ={\rm Pinj}({\mathcal D})$ and so $\overline{\mathcal D}$ is the definable subcategory of ${\rm Mod}\mbox{-}{\mathcal B}$ generated by ${\rm Pinj}({\mathcal D})$.
\end{cor}
\begin{proof} If  $N_{\mathcal B}$ is pure-injective as a ${\mathcal B}$-structure then certainly it is pure-injective as an ${\mathcal A}$-structure, and we have just shown the converse.
\end{proof}

\begin{cor}\label{complissub} For any definable category we have ${\mathcal D} = \overline{\mathcal D}$.
\end{cor}
\begin{proof} By \cite[p.~462]{KraEx} (see \cite[12.2]{PreDefAddCat}), the category ${\mathcal A}= {\rm fun}({\mathcal D})$ is determined by ${\rm Pinj}({\mathcal D})$ (specifically ${\rm fun}({\mathcal D}) = (((({\rm Pinj}({\mathcal D}), {\bf Ab})^{\rm fp})^{\rm op})^{\rm fp})^{\rm op}$) and so, retaining the notation from above, ${\mathcal A}={\mathcal B}$ and hence $\overline{\mathcal D}={\mathcal D}$.
\end{proof}

\begin{theorem}\label{detbypinj} Let ${\mathcal A}$ be a skeletally small abelian category and let ${\mathcal D}={\rm Ex}({\mathcal A},{\bf Ab})$ be the corresponding definable category.  Then ${\mathcal A} = ({\rm Pinj}({\mathcal D}),{\bf Ab})^{\rightarrow \prod}$.  In particular every functor on ${\rm Pinj}({\mathcal D})$ which commutes with direct products and those direct limits which are in ${\rm Pinj}({\mathcal D})$ extends to a unique functor in ${\mathbb D}{\mathbb E}{\mathbb F}$ from ${\mathcal D}$ to ${\bf Ab}$.
\end{theorem}

We remark that we do mean ``unique" (as opposed to unique to natural equivalence), the point being that every object of ${\mathcal D}$ is a (pure) subobject of a pure-injective object.

In order to return to the general case, we set ${\mathcal A}'=({\mathcal D}, {\mathcal C})^{\rightarrow \prod}$ and ${\mathcal B}'= ({\rm Pinj}({\mathcal D}), {\mathcal C})^{\rightarrow \prod}$.  As remarked in the proof of \ref{emblang}, we have a full, faithful and exact embedding ${\mathcal A}'\rightarrow {\mathcal B}'$.  If this were not an equivalence then the corresponding definable categories would not be equivalent - that is, there would be an exact functor $D'$ from ${\mathcal A}'$ to ${\bf Ab}$ which does not extend to an exact functor from ${\mathcal B}'$ to ${\bf Ab}$.  That would be a ${\mathcal C}$-valued model which distinguishes between ${\mathcal A}'$ and ${\mathcal B}'$.  But, by the analogue, \ref{deligne} below, of Deligne's theorem, there would then be an ${\bf Ab}$-valued model which distinguishes ${\mathcal A}$ and ${\mathcal B}$ (in the above notation) - contradicting \ref{detbypinj}.

\begin{example}\label{closepinj}  This example compares two very closely related definable categories from the point of view of this section.  Take ${\mathcal D}$ to be the category of modules over the localisation ${\mathbb Z}_{(p)}$ of ${\mathbb Z}$.  Then every pure-injective $\overline{{\mathbb Z}_{(p)}}$-module is also a pure-injective ${\mathbb Z}_{(p)}$-module and the converse is almost true, the difference being that although ${\mathbb Q}$ is a pure-injective ${\mathbb Z}_{(p)}$-module the corresponding (indecomposable, torsionfree injective) $\overline{{\mathbb Z}_{(p)}}$-module is the direct sum of continuum many copies of ${\mathbb Q}$.  We have the inclusion of ${\rm Mod}\mbox{-}\overline{{\mathbb Z}_{(p)}}$ into ${\rm Mod}\mbox{-}{\mathbb Z}_{(p)}$ but not as a definable subcategory since the former is not closed under pure submodules.  If $r\in \overline{{\mathbb Z}_{(p)}} \setminus {\mathbb Z}_{(p)}$ and we let $G$ be the functor from ${\rm Mod}\mbox{-}\overline{{\mathbb Z}_{(p)}}$ to ${\bf Ab}$, and hence from ${\rm Pinj}({\rm Mod}\mbox{-}{\mathbb Z}_{(p)}$ to ${\bf Ab}$,  which is multiplication by $r$ then this does not extend to a functor on ${\rm Mod}\mbox{-}{\mathbb Z}_{(p)}$.
\end{example}

\subsection{${\rm Ex}({\mathcal A},{\mathcal D})$ is definable when ${\mathcal D}$ is Grothendieck}\label{secexactmorsdef}

We know that if $ {\mathcal A}\in {\mathbb A}{\mathbb B}{\mathbb E}{\mathbb X} $ then $ {\rm Ex}({\mathcal A},{\bf Ab}) $ is a definable category. We show that, more generally, if $ {\mathcal D} $ is a definable category and is abelian, hence Grothendieck, hence locally finitely presented, then $ {\rm Ex}({\mathcal A},{\mathcal D}) $ is again definable.

First we consider the case that $ {\mathcal A} $ is $ {\rm Ab}({\mathcal R}) $ and that $ {\mathcal D} $ is $ {\rm Mod}\mbox{-}{\mathcal S} $ for some skeletally small preadditive categories $ {\mathcal R} $ and $ {\mathcal S}$.

\begin{lemma}\label{exbimod} If $ {\mathcal R}, {\mathcal S} \in {\mathbb P}{\mathbb R}{\mathbb E}{\mathbb A}{\mathbb D}{\mathbb D} $ then $ ({\mathcal R},{\rm Mod}\mbox{-}{\mathcal S})\simeq {\mathcal R}\mbox{-}{\rm Mod}\mbox{-}{\mathcal S} $ the category of $ ({\mathcal R},{\mathcal S})$-bimodules.
\end{lemma}
\begin{proof} An additive functor $ {\mathcal R}\rightarrow {\rm Mod}\mbox{-}{\mathcal S} $ is an $ {\mathcal S}$-module $ M_{{\mathcal S}} $ together with a left $ {\mathcal R}$-module action in which each multiplication by $ r\in {\mathcal R} $ (meaning $r$ is a morphism of ${\mathcal R}$) is a homomorphism in $ {\rm Mod}\mbox{-}{\mathcal S}$, that is, such that the $ {\mathcal R}$- and $ {\mathcal S}$-actions commute.
A natural transformation from one such to the other is given by an $ {\mathcal S}$-module homomorphism $ M_{{\mathcal S}}\rightarrow N_{{\mathcal S}} $ which commutes with each $ r$-action, that is, a homomorphism of bimodules.
\end{proof}

\begin{example}\label{exchains}\marginpar{exchains} Take ${\mathcal S}$ to be the preadditive $S$-category freely generated by the quiver $ \dots \rightarrow \bullet_{n-1} \rightarrow \bullet_n \rightarrow \bullet_{n+1} \rightarrow \dots$ where $S$ is a ring, so ${\rm Mod}\mbox{-}{\mathcal S}$ is the category of chains of $S$-modules (with the arrows going right to left) (of which, note both the category of chain complexes and the category of exact chain complexes are definable subcategories).  Let ${\mathcal R}$ be the preadditive $S$-category generated by $\ast \rightarrow \ast'$.  Then an $({\mathcal R}, {\mathcal S})$-bimodule, a functor from ${\mathcal R}$ to ${\rm Mod}\mbox{-}{\mathcal S}$, is given by two chains of $S$-modules and a morphism between them (meaning that all the squares commute).  So ${\mathcal R}\mbox{-}{\rm Mod}\mbox{-}{\mathcal S}$ is the category of morphisms between chains of $S$-modules, alternatively of representations of the quiver $A_2 = \bullet \rightarrow \bullet$ in the category of complexes of $S$-modules.
\end{example}

Since an additive functor $ {\mathcal R}\rightarrow {\rm Mod}\mbox{-}{\mathcal S} $ is equivalent to an exact functor $ {\rm Ab}({\mathcal R})\rightarrow {\rm Mod}\mbox{-}{\mathcal S} $ (\ref{freeabthm}) we obtain the next corollary.

\begin{cor}\label{exactdef1}  If $ {\mathcal R},{\mathcal S}\in {\mathbb P}{\mathbb R}{\mathbb E}{\mathbb A}{\mathbb D}{\mathbb D} $ then $ {\rm Ex}({\rm Ab}({\mathcal R}),{\rm Mod}\mbox{-}{\mathcal S})\simeq {\mathcal R}\mbox{-}{\rm Mod}\mbox{-}{\mathcal S} = ({\mathcal R}\otimes {\mathcal S}^{\rm op})\mbox{-}{\rm Mod}\simeq {\rm Ex}({\rm Ab}({\mathcal R}\otimes {\mathcal S}^{\rm op}),{\bf Ab})$.
\end{cor}

From this special case we will extend to the general one. We will take the view of $ ({\mathcal R},{\mathcal S})$-bimodules from the proof of the lemma above:  that they are $ {\rm Obj}({\mathcal R})$-indexed collections of objects of $ {\rm Mod}\mbox{-}{\mathcal S} $ linked by the arrows of $ {\mathcal R}$. So, in general, an $ ({\mathcal R},{\mathcal S})$-bimodule is, if we forget the $ {\mathcal R}$-module structure, not an $ {\mathcal S}$-module but a collection of $ {\mathcal S}$-modules. We will refer to these as the {\bf component} $ {\mathcal S}$-modules of the bimodule and, of course, we may make restrictions on these and hence restrictions on the bimodule. In particular, if $ {\mathcal D} $ is a definable subcategory of $ {\rm Mod}\mbox{-}{\mathcal S}$, defined by certain axioms (which specify that certain pp-pairs should be closed) then it makes sense to require that each component of an $ ({\mathcal R},{\mathcal S})$-bimodule lie in $ {\mathcal D}$. Clearly this is a set of conditions in the language of $ ({\mathcal R},{\mathcal S})$-bimodules specifying closure of certain pp-pairs and hence gives a definable subcategory of $ ({\mathcal R},{\mathcal S})\mbox{-}{\rm Mod}$.  We present an example before stating this formally.

\begin{example} In the example above, the category of chain complexes of $S$-modules is definable; let us be more explicit about this.  The language for ${\mathcal S}$-modules has a sort, $s_n$ say, for each integer $n$ and, for each $n$, a constant symbol $0_n$ for the $0$ in sort $n$, a 2-ary function symbol $+_n$ for addition in that sort and, for each $s\in S$ a 1-ary function symbol, let us write it as $f_{s,n}$ for multiplication by $s$ in that sort (though, in practice we would just write it as ``$s$" in formulas).  We also have in the language a function symbol, $f_{d_n}$ say, for the given arrow, $d_n$ say, from $\bullet_n$ to $\bullet_{n+1}$; because we chose to look at right ${\mathcal S}$-modules, that is ${\mathcal S}^{\rm op}$-modules, the symbol $d_n$ will be for a function going from sort $s_{n+1}$ to sort $s_n$.  The condition on a right ${\mathcal S}$-module that the composition $d^2$ be zero is equivalent to closure of the pp-pair $(\exists y_{n+1} (x_{n-1} = f_{d_n}f_{d_{n-1}}y_{n+1})) \, / \, (x_{n-1}=0)$ or, writing it as we would write formulas for modules,  $(\exists y_{n+1} (x_n = y_{n+1}d_nd_{n-1})) \, /\, (x_{n-1} =0)$ (the subscripts to the variables indicate their sort, though in these formulas they are redundant since they are determined by the sorting of the function symbols).  Similarly, exactness of chain complexes can be axiomatised by adding to these pairs the collection of pp-pairs of the form $(x_nd_{n-1}=0)\,/\, (\exists x_{n+1} ( x_n=x_{n+1}d_n))$.

The language for $({\mathcal R}, {\mathcal S})$-bimodules will have its sorts indexed by ${\rm Obj}({\mathcal R}) \times {\rm Obj}({\mathcal S})$; denote these as $s_{\ast n}$, $s_{\ast' n}$ ($n\in {\mathbb Z}$).  So if we take our set of pp-pairs to have a ``$\ast$-copy" and a ``$\ast '$-copy" of each of the above pp-pairs then these will cut out the subcategories of $A_2$-representations in chain complexes (respectively, in exact complexes) of $S$-modules.

(Of course one need not use the same set of ${\mathcal S}$-module axioms on each component ${\mathcal S}$-module so one can see further ways of specifying definable subcategories of $({\mathcal R}, {\mathcal S})\mbox{-}{\rm Mod}$.)
\end{example}

\begin{prop}\label{exactdef2} Suppose $ {\mathcal R}\in {\mathbb P}{\mathbb R}{\mathbb E}{\mathbb A}{\mathbb D}{\mathbb D} $ and $ {\mathcal D}$ is an abelian definable category. Then $ {\rm Ex}({\rm Ab}({\mathcal R}),{\mathcal D})\simeq {\rm Ex}({\mathcal B},{\bf Ab}) $ for some $ {\mathcal B}\in {\mathbb A}{\mathbb B}{\mathbb E}{\mathbb X}$.
\end{prop}
\begin{proof} Suppose that $ {\mathcal D} $ is a definable subcategory of $ {\rm Mod}\mbox{-}{\mathcal S}$. Then, as argued above, an additive functor from $ {\mathcal R} $ to $ {\mathcal D} $ is just an $ ({\mathcal R},{\mathcal S})$-bimodule which belongs to a certain definable subcategory of the functor category $ ({\mathcal R}\otimes {\mathcal S}^{\rm op})\mbox{-}{\rm Mod} $. That is, $ ({\mathcal R},{\mathcal D}) $ is a definable subcategory of $ ({\mathcal R},{\mathcal S})\mbox{-}{\rm Mod}$, hence is a definable category and so $ {\rm Ex}({\rm Ab}({\mathcal R}),{\mathcal D})\simeq ({\mathcal R},{\mathcal D})\simeq {\rm Ex}({\mathcal B},{\bf Ab}) $ for some $ {\mathcal B}\in {\mathbb A}{\mathbb B}{\mathbb E}{\mathbb X}$.
\end{proof}

The final step is to replace $ {\rm Ab}({\mathcal R}) $ by a general $ {\mathcal A}\in {\mathbb A}{\mathbb B}{\mathbb E}{\mathbb X} $ but that, as noted in \ref{abisquot}, has the form $ {\rm Ab}({\mathcal R})/{\mathcal T} $ for some $ {\mathcal R}\in {\mathbb P}{\mathbb R}{\mathbb E}{\mathbb A}{\mathbb D}{\mathbb D} $ and ${\mathcal T}\in {\rm Ser}({\rm Ab}({\mathcal R}))$. Then the quotient functor $ {\rm Ab}({\mathcal R})\rightarrow {\rm Ab}({\mathcal R})/{\mathcal T}={\mathcal A} $ induces an embedding of $ {\rm Ex}({\mathcal A},{\mathcal D}) $ as a subcategory of $ {\rm Ex}({\rm Ab}({\mathcal R}),{\mathcal D})$, namely as the full subcategory on those functors which annihilate the Serre subcategory ${\mathcal T}$. We have just seen than $ {\rm Ex}({\rm Ab}({\mathcal R}),{\mathcal D}) $ is definable and it is easy to see that the condition of annihilating $ {\mathcal T} $ is a definable one:  for each object of (a generating set of) ${\mathcal T}$ choose a pp-pair which defines it (we are using, as we did near the beginning of Section \ref{secabexfinacc}, the description of these exact functors as pp-pairs). Thus we have an embedding of $  {\rm Ex}({\mathcal A},{\mathcal D}) $ as a definable subcategory of $ {\rm Ex}({\rm Ab}({\mathcal R}),{\mathcal D}) $ and so $ {\rm Ex}({\mathcal A},{\mathcal D}) $ is indeed definable.

\begin{theorem}\label{exactdef} Suppose that $ {\mathcal A} $ is a skeletally small abelian category  and that $ {\mathcal D} $ is a definable category which is abelian. Then $ {\rm Ex}({\mathcal A},{\mathcal D}) $ is a definable category, in particular is equivalent to $ {\rm Ex}({\mathcal B},{\bf Ab}) $ for some small abelian category $ {\mathcal B}$.
\end{theorem}

In the topos-analogy view of the next section, one would say that if ${\mathcal E} = {\rm Ex}({\mathcal A}, {\bf Ab})$ is a definable category and ${\mathcal D}$ is an abelian = locally finitely presented Grothendieck definable category then ${\rm Ex}({\mathcal A}, {\mathcal D})$ is the category of ${\mathcal D}$-models of the regular theory with (the additive version of) classifying pretopos ${\mathcal A}$, as opposed to the category ${\mathcal E}$ of ${\bf Ab}$-models.

\section{Locally coherent additive categories}\label{seccoh}

Recall that an object $A$ of an additive category ${\mathcal A}$ is {\bf finitely presented} if the representable functor $(A,-):{\mathcal A} \rightarrow {\bf Ab}$ commutes with direct limits.  Set ${\mathcal A}^{\rm fp}$ to be the full subcategory of finitely presented objects of ${\mathcal A}$.  We say that $A$ is {\bf coherent} if it is finitely presented and each of its finitely generated subobjects is finitely presented.  An abelian category ${\mathcal G}$ with direct limits is {\bf locally coherent} iff it has a generating set consisting of coherent objects.  Such a category is, in particular, finitely accessible, hence (\cite[2.4]{CBlfp}) Grothendieck.  It is easy to check that in such a category finitely presented = coherent.  The category ${\rm Mod}\mbox{-}{\mathcal R}$ is locally coherent iff ${\mathcal R}$ is {\bf right coherent}.

The connection between ${\mathbb A}{\mathbb B}{\mathbb E}{\mathbb X}$ and ${\mathbb C}{\mathbb O}{\mathbb H}$ goes back to Gabriel \cite[11.4 Thm.~1]{Gab} (for locally noetherian categories) and Roos \cite[2.2]{Roo} (in general), also see \cite{ObRoh}.

\begin{theorem}\label{abtofromloccoh} If ${\mathcal A}$ is a skeletally small abelian category then ${\rm Lex}({\mathcal A}^{\rm op}, {\bf Ab}) \simeq {\rm Flat}\mbox{-}{\mathcal A} \simeq {\rm Ind}({\mathcal A})$ is a locally coherent Grothendieck category ${\mathcal G}$ with ${\mathcal G}^{\rm fp} \simeq {\mathcal A}$.

If ${\mathcal G}$ is a locally coherent Grothendieck category then ${\mathcal G}^{\rm fp}$ is abelian and ${\mathcal G} \simeq {\rm Lex}(({\mathcal G}^{\rm fp})^{\rm op}, {\bf Ab})$.
\end{theorem}

The Ind-completion, ${\rm Ind}({\mathcal A})$, of ${\mathcal A}$ is the free extension of ${\mathcal A}$ to a category with direct limits (the objects can be defined to be equivalence classes of directed diagrams in ${\mathcal A}$, see, e.g., \cite[Exp.~1, \S 8]{SGA4}, \cite[\S VI.1]{Joh}, \cite[Chpt.~6]{KashSchap}).  To see why ${\rm Flat}\mbox{-}{\mathcal A} \simeq {\rm Ind}({\mathcal A})$ recall that the flat right ${\mathcal A}$-modules are the direct limits of (finitely presented) projective ${\mathcal A}$-modules and those are the images of objects of ${\mathcal A}$ under the Yoneda embedding.

The morphisms of $ {\mathbb C}{\mathbb O}{\mathbb H}$ are adjoint pairs of functors where the left adjoint is exact and preserves coherence of objects.  Precisely, a morphism $f:{\mathcal G} \rightarrow {\mathcal H}$ is given by a morphism $f^\ast:{\mathcal H} \rightarrow {\mathcal G}$ and a morphism $f_\ast:{\mathcal G} \rightarrow {\mathcal H}$ such that $(f^\ast,f_\ast)$ is an adjoint pair with the left adjoint $f^\ast$ being (left) exact and with $f^\ast {\mathcal H}^{\rm fp} \subseteq {\mathcal G}^{\rm fp}$.  We refer to these as {\bf coherent morphisms} and a typical notation is $f:{\mathcal G} \rightarrow {\mathcal H}$, which expands to $\xymatrix{{\mathcal G} \ar@/_/[r]_{f_\ast}  & {\mathcal H} \ar@/_/[l]_{f^\ast}}$.  This is an additive analogue of the definition of a geometric morphism, e.g.~\cite[A4.1.1]{Joh1}, with the added finiteness condition that $f^\ast$ preserve coherent objects.  Such pairs of functors have already been considered by Krause, see \cite[\SS 9-11]{KraFun}; in particular, \cite[6.7]{KraFun}, the functor $f_\ast$ also preserves direct limits (as well as absolutely pure and injective objects - that is, the corresponding definable categories and their pure-injective objects, the absolutely pure objects being the exact functors, see, e.g., \cite[\S 11]{PreDefAddCat}).

The 2-arrows of $ {\mathbb C}{\mathbb O}{\mathbb H}$ are the natural transformations.  More precisely, if $f,g:{\mathcal G}\rightarrow {\mathcal H}$ are morphisms, then a 2-arrow from $f=(f^\ast, f_\ast)$ to $g=(g^\ast, g_\ast)$ is a natural transformation $\tau^\ast:f^\ast \rightarrow g^\ast$ equivalently, see below, a natural transformation $\tau_\ast:g_\ast \rightarrow f_\ast$.

\begin{lemma}\label{adjnattrans}\marginpar{adjnattrans} Suppose that $f=(f^\ast,f_\ast)$ and $g=(g^\ast,g_\ast)$ are two adjoint pairs of functors from ${\mathcal G}$ to ${\mathcal H}$.  Then there is a natural bijection between natural transformations $\tau^\ast: f^\ast \rightarrow g^\ast$ and natural transformations $\tau_\ast:g_\ast \rightarrow f_\ast$.  This bijection commutes with composition of natural transformations as well as with horizontal and vertical composition of natural transformation with functors.
\end{lemma}
\begin{proof} Here is how to get from $\tau^\ast$ to its correspondent $\tau_\ast$.  More details are at \cite[V.2.1]{MitBk}.

We have, for $H\in {\mathcal H}$, the component $f^\ast H \xrightarrow{\tau^\ast_H} g^\ast H$.  We apply this with $H=g_\ast G$ and proceed as follows, where $\epsilon_g$ denotes the counit of $g$ and $\eta_f$ denotes the unit of $f$:

\noindent $f^\ast g_\ast G \xrightarrow{\tau^\ast_{g_\ast G}} g^\ast g_\ast G$ and $g^\ast g_\ast G \xrightarrow{(\epsilon_g)_G} G$

\noindent hence

\noindent $f^\ast g_\ast G \xrightarrow{(\epsilon_g)_G \tau^\ast_{g_\ast G}} G$

\noindent hence

$f_\ast f^\ast g_\ast G \xrightarrow{f_\ast (\epsilon_g)_G f_\ast \tau^\ast_{g_\ast G}} f_\ast G$ together with $g_\ast G \xrightarrow{(\eta_f)_{g_\ast G}} f_\ast f^\ast g_\ast G$

\noindent giving

\noindent $g_\ast G \xrightarrow{f_\ast (\epsilon_g)_G f_\ast \tau^\ast_{g_\ast G} (\eta_f)_{g_\ast G}} f_\ast G$

\noindent We set $\tau_\ast =f_\ast (\epsilon_g)_G .f_\ast \tau^\ast_{g_\ast G} .(\eta_f)_{g_\ast G}$.
\end{proof}

The correspondence between ${\mathbb A}{\mathbb B}{\mathbb E}{\mathbb X}$ and ${\mathbb C}{\mathbb O}{\mathbb H}$ appears in \cite{PreRajShv} but without details (there the focus is on the anti-equivalence between ${\mathbb A}{\mathbb B}{\mathbb E}{\mathbb X}$ and ${\mathbb D}{\mathbb E}{\mathbb F}$; the suggestion of considering the associated locally coherent categories came from the referee of that paper).  Here we give the some details (since I am not aware of any reference for them).

\begin{theorem}\label{abexcohequiv} There is a natural anti-equivalence of 2-categories between ${\mathbb C}{\mathbb O}{\mathbb H}$ and ${\mathbb A}{\mathbb B}{\mathbb E}{\mathbb X}$, as described above.
\end{theorem}

We have already seen, in \ref{abtofromloccoh}, the correspondence on objects.
On morphisms, the anti-equivalence between ${\mathbb A}{\mathbb B}{\mathbb E}{\mathbb X}$ and ${\mathbb C}{\mathbb O}{\mathbb H}$ is as follows.

\begin{prop}\label{mors} (\cite[10.1]{KraFun}) To each arrow $f_0\in {\rm Ex}({\mathcal A}, {\mathcal B})$ there corresponds the coherent morphism $f=(f^\ast,f_\ast):{\mathcal H}={\rm Ind}({\mathcal B}) \rightarrow {\rm Ind}({\mathcal A})={\mathcal G}$ which has $f_\ast:{\mathcal H}={\rm Lex}({\mathcal B}^{\rm op}, {\bf Ab}) \rightarrow {\rm Lex}({\mathcal A}^{\rm op}, {\bf Ab})={\mathcal G}$ just precomposition with $f_0^{\rm op}$ and has $f^\ast ={\rm Ind}(f_0)$ (the Ind-construction is defined in an obvious way on functors).

In the other direction we take an arrow $(f^\ast,f_\ast):{\mathcal H} \rightarrow {\mathcal G}$ to $f^\ast\upharpoonright {\mathcal G}^{\rm fp}$.
\end{prop}
\begin{proof}  The second statement is immediate since if $f^\ast$ is exact then so is its restriction.  For the first statement, the only part that takes some checking is that $f$ is a coherent morphism.

First, we show that $ {\rm Ind}(f_0) $ is exact. Given an exact sequence $ 0\rightarrow G'\rightarrow G\rightarrow G''\rightarrow 0 $ in $ {\mathcal G} $ write $ G $ as a direct limit of finitely presented objects, $ G=\varinjlim G_\lambda $, let $ G'_\lambda  $ be the pullback of $ G_\lambda \rightarrow G $ and $ G'\rightarrow G$, write $ G'_\lambda  $ as a direct limit of finitely generated (hence finitely presented) subobjects $ G'_{\lambda \mu } $ and set $ G_{\lambda \mu }''={\rm coker}(G_{\lambda ,\mu }'\rightarrow G_\lambda ) $ and note the induced canonical map $ G_{\lambda \mu }''\rightarrow G''_\lambda $. All these exact sequences fit together into a directed system in $ {\mathcal G}^{\rm fp}$, the colimit of which is the original exact sequence. Since $ f_0 $ is exact, the component exact sequences $ 0\rightarrow G'_{\lambda \mu }\rightarrow G_\lambda \rightarrow G''_{\lambda \mu }\rightarrow 0 $ are taken by it to exact sequences in $ {\mathcal H}$. These exact sequences form a directed system in $ {\mathcal H} $ and so the direct limit is exact and, by definition of $ f^\ast  $ as $ {\rm Ind}(f_0)$, that direct limit is $  0\rightarrow f^\ast  G'\rightarrow f^\ast  G\rightarrow f^\ast  G''\rightarrow 0$, so $ f^\ast  $ is exact. Of course, by definition $ f^\ast  $ takes $ {\mathcal G}^{\rm fp} $ to $ {\mathcal G}^{\rm fp}$.

We must show that $ {\rm Ind}(f_0) $ is left adjoint to $ f_\ast  $, so take $ H\in {\mathcal H} $ and $ G\in {\mathcal G}$. Write $ G $ as a direct limit, $ G=\varinjlim _\lambda G_\lambda $, of finitely presented objects in $ {\mathcal G} $ where we have $ G_\lambda =(-,A_\lambda ) $ for some directed system $ \{ A_\lambda \}_\lambda  $ of objects of $ {\mathcal A}$.  Then:

\noindent $\big({\rm Ind}(f_0)G,H\big) = \big({\rm Ind}(f_0)(\varinjlim _\lambda G_\lambda ),H\big) = \big({\rm Ind}(f_0)(\varinjlim _\lambda (-,A_\lambda )),H\big)$

\noindent $ = \big(\varinjlim _\lambda {\rm Ind}(f_0)(-,A_\lambda ),H\big)$ by definition of $ {\rm Ind}(f_0)$

\noindent $ = \big(\varinjlim _\lambda f_0(-,A_\lambda ), H\big) = \big(\varinjlim _\lambda (-,f_0A_\lambda ), H\big)$

\noindent $ = \varprojlim_\lambda \big((-,f_0A_\lambda ),H\big) = \varinjlim _\lambda H(f_0A_\lambda ) $ (recall that $ H $ is a functor on $ {\mathcal B}^{\rm op}$)

\noindent $ =\varprojlim_\lambda f_\ast HA_\lambda  $ ($f_\ast H $ is a functor on $  {\mathcal A}^{\rm op} $ but now we convert the action {\it via} Yoneda to an action on $ {\mathcal A}$, expressed within $ {\mathcal G}$)

\noindent $ = \varprojlim_\lambda \big((-,A_\lambda ),f_\ast H\big) $ (by Yoneda)

\noindent $ = \big(\varinjlim _\lambda (-,A_\lambda ), f_\ast H\big) = (G,f_\ast H)$.
\end{proof}

One may also check, for \ref{abexcohequiv}, the correspondence between 2-arrows.

\begin{example} Let $i_0$ be the inclusion of the category, ${\rm reg}\mbox{-}k\widetilde{A_1}$, of finite-dimensional regular modules over the Kronecker algebra $k\widetilde{A_1}$ (or over any tame hereditary finite-dimensional algebra) into the category ${\rm mod}\mbox{-}k\widetilde{A_1}$ of finite-dimensional modules.  This is exact, so induces a coherent morphism as above, with the left adjoint being the inclusion of all regular modules (i.e.~direct limits of finite-dimensional regular modules) and the right adjoint taking an arbitrary $k\widetilde{A_1}$-module $M$ to the module which, in the notation and terminology of \cite[p.~326]{RinInf} (also see \cite{ReiRin}) is ${\mathcal T}(M)/{\mathcal I}(M)$ where ${\mathcal T}(M)$, the torsion submodule of $M$, is the sum of all regular and preinjective submodules of $M$, modulo the sum, ${\mathcal I}(M)$ of all preinjective submodules of $M$ which, since the ring is tame hereditary, will be a direct summand (\cite[3.7]{RinInf}).

We explore this example a bit further.  The category ${\rm reg}\mbox{-}k\widetilde{A_1}$ is a direct sum of uniserial abelian categories indexed by ${\mathbb P}^1(k)$, that indexed by an irreducible monic polynomial $f\in k[T]$ being equivalent to the category of finitely generated torsion modules over $k[T]_{(f)}$.  The ${\rm Ind}$-completion of this category is, by \cite[2.1(iii)]{LenzTrans}, the full subcategory, ${\rm Reg}\mbox{-}k\widetilde{A_1}$, on those $k\widetilde{A_1}$-modules every finite-dimensional submodule of which is regular, and is equivalent to the direct product of the categories of torsion modules over the various localisations of $k[T]$.

The functor $(i_0)_\ast:M\mapsto {\mathcal T}(M)/{\mathcal I}(M)$ takes the generic module $G$, since it is a direct limit of preprojective modules and since, by \cite[6.7]{KraFun}, $(i_0)_\ast$ commutes with direct limits, to $0$.  Also, given any preinjective module $I$ there is, by \cite[\S\S 6,7]{ReiRin}, an exact sequence $0\rightarrow V'\rightarrow V \rightarrow I \rightarrow 0$ with $V'$ a direct sum of copies of $G$ and $V$ a direct sum of copies of $G$ and Pr\"{u}fer modules.  The image under $(i_0)_\ast$ of this exact sequence has, therefore, just the one possibly non-zero term $(i_0)_\ast V$; but this is $0$ only if the sequence is split - which need not be the case - and so $(i_0)_\ast$ is not exact.

The definable category corresponding to $k\widetilde{A_1}\mbox{-}{\rm mod}$ is, by \ref{cohexgen}, ${\rm Proj}\mbox{-}k\widetilde{A_1}$.  That corresponding to ${\rm reg}\mbox{-}k\widetilde{A_1}$ is ${\rm Ex}({\rm reg}\mbox{-}k\widetilde{A_1}, {\rm Ab})$ which, by the decomposition of ${\rm reg}\mbox{-}k\widetilde{A_1}$, is equivalent to the ${\mathbb P}^1(k)$-indexed product of categories of the form ${\rm Ex}({\rm tors}\mbox{-}k[T]_{(f)}, {\bf Ab})$ which (cf.~a similar computation in \cite{PreAxtFlat}) may be computed to be equivalent to the category of reduced flat $k[T]_{(f)}$-modules.  Putting these together, we deduce that ${\rm Ex}({\rm reg}\mbox{-}k\widetilde{A_1}, {\rm Ab})$ may be taken to be the category of ``reduced torsionfree'' $k\widetilde{A_1}$-modules.  The elementary dual of this category consists of the direct sums of Pr\"{u}fer (``torsion modulo divisible", see \cite[p.~326]{RinInf}) $k\widetilde{A_1}$-modules.
\end{example}

Here is the action of duality on $ {\mathbb C}{\mathbb O}{\mathbb H}$: if ${\mathcal G} = {\rm Lex}(({\mathcal G}^{\rm fp})^{\rm op}, {\bf Ab})$ is a locally coherent Grothendieck category then the duality induced on $ {\mathbb C}{\mathbb O}{\mathbb H}$ will take it to ${\rm Lex}({\mathcal G}^{\rm fp}, {\bf Ab})$ which is the {\bf conjugate} category in the sense of Roos (\cite[p.~204]{Roo}) and which is denoted $\widetilde{\mathcal G}$.  The action on morphisms is as follows:  if $(f^\ast,f_\ast):{\mathcal H} \rightarrow {\mathcal G}$ is the coherent morphism induced by $f_0\in {\rm Ex}({\mathcal A}, {\mathcal B})$ then its ``conjugate'' is the coherent morphism induced by $f_0^{\rm op}\in {\rm Ex}({\mathcal A}^{\rm op}, {\mathcal B}^{\rm op})$, namely $(\widetilde{f}^\ast,\widetilde{f}_\ast):\widetilde{\mathcal H} \rightarrow \widetilde{\mathcal G}$ with $\widetilde{f}_\ast$ being precomposition with $f_0$ and $\widetilde{f}^\ast = {\rm Ind}(f_0^{\rm op})$.  This follows directly from the definition of the involution on ${\mathbb A}{\mathbb B}{\mathbb E}{\mathbb X}$ and the description of the equivalence in \ref{abexcohequiv}.

One might note any locally coherent Grothendieck category ${\mathcal G}$ also is an object of $ {\mathbb D}{\mathbb E}{\mathbb F}$, however, its elementary dual ${\mathcal G}^{\rm d}$ (Section \ref{secdefcats}) need not (though in some cases will) coincide with its conjugate:  for example, if $R$ is right coherent but not left coherent then the conjugate category of ${\rm Mod}\mbox{-}R$ (which is described in \cite{Roo}) cannot be its elementary dual $R\mbox{-}{\rm Mod}$ since the latter is not locally coherent.

\vspace{4pt}

If $ {\mathcal G} $ is a locally coherent Grothendieck category then a full subcategory $ {\mathcal G}' $ is a {\bf Giraud subcategory} if it is {\bf reflective}, that is, the inclusion $ i:{\mathcal G}'\rightarrow {\mathcal G} $ has a left adjoint and that left adjoint is (left) exact. These are exactly the subcategories of $ {\mathcal G} $ of the form $ {\mathcal G}_\tau  $, obtained by localising $ {\mathcal G} $ at a hereditary torsion theory $ \tau  $ and then composing the localisation functor $ Q_\tau :{\mathcal G} \rightarrow {\mathcal G}_\tau $  with its right adjoint (which embeds the localised category back into $ {\mathcal G}$). This is the additive analogue of sheafification in the topos/categories of sheaves context.

\begin{prop}\label{giraud} Let $ \tau  $ be a hereditary torsion theory on the locally coherent Grothendieck category $ {\mathcal G}$. Then the following conditions are equivalent.

\noindent (i) $ \xymatrix{{\mathcal G}_\tau  \ar@/_/[r]_{i} & {\mathcal G} \ar@/_/[l]_{Q_\tau }}$ is a coherent morphism from $ {\mathcal G}_\tau  $ to $ {\mathcal G}$;

\noindent (ii) the localisation functor $ Q_\tau  $ takes finitely presented objects of $ {\mathcal G} $ to finitely presented objects of $ {\mathcal G}_\tau $;

\noindent (iii) $  \tau  $ is of finite type (equivalently, since $ {\mathcal G} $ is locally coherent, an elementary torsion theory in the sense of \cite{Prelfp}, see \cite[\S11.1.3]{PreNBK})
\end{prop}
\begin{proof} The equivalence of (i) and (ii) is just the definitions. That (iii) implies these is part of \cite[11.1.26]{PreNBK}.  We could finish by appealing to \cite[6.7]{KraFun} and, e.g., \cite[11.1.12]{PreNBK} but a direct proof is straightforward, as follows.  Suppose that $ \tau  $ is not of finite type. Then there is $ G\in {\mathcal G}^{\rm fp} $ and $ G'\leq G $ which is $ \tau $-dense (i.e.~$G/G'$ is torsion) but which has no finitely generated $ \tau $-dense subobject (we do assume some acquaintance with torsion theory here - see \cite{Ste}, though what is in \cite[Chpt.~11]{PreNBK} should suffice). So $ G'=\sum_\lambda {\rm cl}^{G'}_\tau G_\lambda  $ - the directed sum of the $ \tau $-closures in $ G' $ of its finitely generated subobjects - and for each $ \lambda  $ we have $ {\rm cl}^{G'}_\tau G_\lambda <G'$.

Since $ {\mathcal G} $ is locally coherent each $ G_\lambda  $ is finitely presented so, by hypothesis, $ Q_\tau G_\lambda  $ is a finitely presented object of $ {\mathcal G}_\tau $. Since $ Q_\tau  $ is a left adjoint it commutes with directed colimits so we have $ Q_\tau G =Q_\tau G' =Q_\tau \varinjlim _\lambda {\rm cl}^{G'}_\tau (G_\lambda ) = \varinjlim _\lambda Q_\tau {\rm cl}^{G'}_\tau (G_\lambda ) = \varinjlim _\lambda Q_\tau {\rm cl}^G_\tau (G_\lambda ) $ (since $ G' $ is $ \tau $-dense in $ G$) $ = \varinjlim _\lambda Q_\tau G_\lambda $. If $ Q_\tau G $ were finitely presented then there would be $ \lambda  $ such that $ Q_\tau G_\lambda =Q_\tau G $ and hence such that $G_\lambda$ is $\tau$-dense in $G$ so, in particular, $ G_\lambda $ would be $ \tau $-dense in $ G' $ - a contradiction.
\end{proof}

\begin{example}\label{finiteab} Take ${\mathcal A}$ to be the category of finite abelian groups and  ${\mathcal B}$  to be the category of finitely generated abelian groups, so $ {\rm Ind}(i)$, where $i$ is the inclusion of  ${\mathcal A}$ in  ${\mathcal B}$, is the inclusion of the category,  ${\rm Tors}\mbox{-}{\mathbb Z}$, of torsion abelian groups in  ${\rm Mod}\mbox{-}{\mathbb Z}$, with right adjoint the torsion functor  $\tau$.  This give a coherent morphism from  ${\rm Mod}\mbox{-}{\mathbb Z}$ to  ${\rm Tors}\mbox{-}{\mathbb Z}$:  $\xymatrix{{\rm Mod}\mbox{-}{\mathbb Z} \ar@/_/[r]_{\tau} & {\rm Tors}\mbox{-}{\mathbb Z} \ar@/_/[l]_{{\rm Ind}(i)}}$. Note that in this example it is the case that  $\tau$  preserves finitely presented objects; this is far from being the case for a general coherent morphism (consider, for instance the example $\xymatrix{{\rm Mod}\mbox{-}{\mathbb Q} \ar@/_/[r]_i & {\rm Mod}\mbox{-}{\mathbb Z} \ar@/_/[l]_{Q_\tau = -\otimes {\mathbb Q}}}$ of \ref{giraud}).
\end{example}

\subsection{The additive version of the classifying topos}\label{sectopos}

In this section we draw out explicitly the analogies between what we have here and the situation seen in topos theory.  Recall (\cite{SGA4}, \cite{Joh1}, \cite{MacMoe}) that if ${\mathcal C}$ is a (skeletally) small category then a contravariant functor from ${\mathcal C}$ to the category, ${\bf Set}$, of sets is referred to as a {\bf presheaf} on ${\mathcal C}$.  In the classical case - that of a presheaf on a topological space - the category ${\mathcal C}$ is the collection of open subsets of the space, ordered by inclusion and thereby conceived of as a poset-type category (indeed a locale).  The category, $({\mathcal C}^{\rm op}, {\bf Set})$, of such functors is the non-additive correspondent to the category, ${\rm Mod}\mbox{-}{\mathcal R} =({\mathcal R}^{\rm op}, {\bf Ab})$, of right modules over a (skeletally) small preadditive category ${\mathcal R}$.

A Grothendieck topology on ${\mathcal C}$ is a notion of ``covering" (generalising the notion of open cover in topology);  ${\mathcal C}$ together with a Grothendieck topology is termed a site.  The additive analogue is a hereditary torsion theory on ${\rm Mod}\mbox{-}{\mathcal R}$ (presented as a system of dense subobjects of the generating projectives $(-,P)$ for $P\in {\mathcal R}$).  A sheaf for a site is a presheaf which satisfies the sheaf conditions (existence and uniqueness of patchings) for that site.  There is a sheafification functor on the category of presheaves and the resulting category of sheaves has an embedding, right adjoint to sheafification, as a full subcategory of the category of presheaves.  The additive analogue of a sheaf is a module which is torsionfree and divisible with respect to the torsion theory.  There is a localisation (=sheafification) functor on the category of modules and the resulting quotient category has an embedding, right adjoint to localisation, as a full subcategory of the category of modules.

The analogy is wide, extends right down to technical details, and has been known for a long time.  An algebraist need only skim through the relevant sections of \cite{Bor3} and a category-theorist through \cite{Ste} to see the parallels.  This paper, motivated by the model theory of additive structures, is, in part, drawing out the analogy further and I hope that it is a useful thing to do and that it will give further insight and results.  So far as I am aware there is no way of directly deducing the additive version of results from the Set-version.  Of course, one can say that the former is an enriched version of the latter but it is not clear to me that allows one to avoid having to write down the details to obtain the additive version.

A Grothendieck topos is the category of sheaves on a site; the additive analogue is the notion of a Grothendieck abelian category.  There is a notion of coherent topos - this can be defined as a topos which can be obtained from a presheaf category by sheafification at a finite-type site (see \cite[\S IX.11]{MacMoe}).  A finite-type site is one where the underlying category ${\mathcal C}$ has finite limits and the topology has a basis of finite covering families.  Analogously a locally coherent Grothendieck category is one which can be obtained by a finite-type localisation of a locally coherent module category, and a module category is locally coherent iff it is (Morita) equivalent to one of the form ${{\rm Mod}\mbox{-}\mathcal R}$ where ${\mathcal R}$ has finite weak limits.  The last statement is almost in \cite{CBlfp} and the former statement almost in \cite[\S\S 11.1.3,11.1.4]{PreNBK}.  Thus ${\mathbb C}{\mathbb O}{\mathbb H}$ is the additive analogue of the 2-category of coherent toposes.

As for ${\mathbb A}{\mathbb B}{\mathbb E}{\mathbb X}$, there is a general notion of {\bf coherent} object: it is an object $C$ such that whenever one has a pullback

$\xymatrix{P \ar@{.>}[r] \ar@{.>}[d] & B \ar[d] \\ B' \ar[r] & C}$

\noindent where $B, B'$ are finitely generated ($B$ is {\bf finitely generated} if the representable functor $(B,-)$ commutes with directed colimits of monomorphisms) then $P$ also is finitely generated.  If ${\mathcal E}$ is a locally coherent topos then the full subcategory of coherent objects forms what is termed a pretopos \cite[\S A1.4]{Joh1}; such a skeletally small, generating, subcategory is the analogue of ${\mathcal G}^{\rm fp}$ when ${\mathcal G} \in {\mathbb C}{\mathbb O}{\mathbb H}$.  Thus the pretoposes which so arise are the analogues of skeletally small abelian categories.

The analogue of ${\mathbb D}{\mathbb E}{\mathbb F}$ arises through the link between toposes and mathematical logic.  There is a very general notion of ``geometric theory" (see \cite{Joh1}, \cite{MacMoe}); if one stays with finitary logic (as is usual in model theory) then one is dealing with ``coherent" theories.  In fact, the strict analogue of the model theory of the additive situation, at least if we want our categories of models to have finite products, is regular logic in the sense of  \cite[\S D1.1]{Joh1}, see \cite[\S 2.5]{Bridge}.  If $T$ is a geometric theory then there is a corresponding ``classifying topos", ${\mathcal B}[T]$ which can be defined by the universal property that if ${\mathcal E}$ is any Grothendieck topos then ${\rm Geom}({\mathcal E}, {\mathcal B}[T]) \simeq {\rm Mod}_{\mathcal E}T$, where ${\rm Geom}({\mathcal E}, {\mathcal F})$ denotes the set, indeed category, of geometric morphisms from ${\mathcal E}$ to ${\mathcal F}$ (these are the analogue of the morphisms in ${\mathbb C}{\mathbb O}{\mathbb H}$) and ${\rm Mod}_{\mathcal E}(T)$ is the category of models of the theory of $T$ in ${\mathcal E}$ (the classical notion of a ${\bf Set}$-model of a theory can be extended for geometric theories to that of a model in any Grothendieck topos).  The category, ${\rm Mod}_{\bf Set}(T)$, of ${\bf Set}$-models parallels the category ${\mathbb D}{\mathbb E}{\mathbb F}$ of definable categories - which are, indeed, categories of ${\bf Ab}$-models of (regular) additive theories.  The analogous equivalence in our context would read ${\mathbb C}{\mathbb O}{\mathbb H}({\mathcal G}, {\rm Fun}({\mathcal D}))\simeq {\rm Mod}_{\mathcal G}({\rm Th}({\mathcal D}))$ where ${\mathcal D}$ is any definable category, ${\rm Th}({\mathcal D})$ is its theory, and where the right-hand side can be defined through this equivalence or, more naturally in view of \ref{defisexfun}, by defining ${\rm Mod}_{\mathcal G}({\rm Th}({\mathcal D}))$ to be ${\rm Ex}({\rm fun}({\mathcal D}), {\mathcal G})$ (see Section \ref{secexactmorsdef}) - the category of ${\mathcal G}$-models of the theory of ${\mathcal D}$.  Thus ${\rm Fun}({\mathcal D})$ plays the role of the classifying topos for (the theory of) ${\mathcal D}$.  We prefer to say things in terms of definable categories ${\mathcal D}$ rather than theories especially since the former are (``Morita"-)invariant objects whereas there are many literally different, though equivalent, theories with ${\mathcal D}$ as the category of models.  It should be noted, though, that the category ${\rm Fun}({\mathcal D})$ can be constructed from any theory $T$ whose category of models is ${\mathcal D}$, essentially in the same way that the classifying topos is constructed syntactically (see, e.g.~\cite[\S 10.5]{MacMoe}).

For more on the parallels see the references already mentioned and also the work (e.g.~\cite{CarLat, CarSites, CarUnif}) of Caramello which heavily exploits the topos analogue of Morita equivalence and the usefulness of a topos having more than one presentation as a category of sheaves on a site.

\subsection{Points of locally coherent categories}

By a point of a topos $ {\mathcal E} $ is meant a geometric morphism $ f:{\bf Set}\rightarrow {\mathcal E}$, that is, an adjoint pair of functors  $\xymatrix{{\bf Set} \ar@/_/[r]_{f_\ast} & {\mathcal E} \ar@/_/[l]_{f^\ast}}$ with the left adjoint $ f^\ast $ left exact, meaning that it preserves finite limits. In our situation $ {\bf Ab} $ replaces $ {\bf Set} $ and we imposed the additional finiteness requirement that $ f^\ast $ should preserve coherent objects, though we relax that here so as to make a better comparison, and will refer to {\bf coherent points} (those that preserve coherent objects) and ({\bf geometric}) {\bf points} (without that condition) $\xymatrix{{\bf Ab} \ar@/_/[r]_{f_\ast} & {\mathcal G} \ar@/_/[l]_{f^\ast}}$.  By \ref{abexcohequiv}/\ref{mors}, to give a coherent point of a locally coherent Grothendieck category $ {\mathcal G} $ it is equivalent to give an exact functor $ {\mathcal G}^{\rm fp}\rightarrow {\rm mod}\mbox{-}{\mathbb Z} $.

If we set ${\mathcal D} ={\rm Ex}({\mathcal G}^{\rm fp}, {\bf Ab})$ (so ${\mathcal G}={\rm Fun}({\mathcal D})$) then each object $D\in {\mathcal D}$ gives the exact functor ${\rm ev}_D : {\mathcal G}^{\rm fp}\rightarrow {\rm Mod}\mbox{-}{\mathbb Z}$ and the requirement that the image of this functor (playing the role of ``$f_0$" in \ref{mors}) be contained in ${\rm mod}\mbox{-}{\mathbb Z}$ is the requirement that each pp-sort $(\phi/\psi)(D)$ be a finitely generated abelian group.  Of course, there might be few or even no such objects; a simple example is given by taking ${\mathcal D}$ (or ${\mathcal G}$) to be ${\rm Mod}\mbox{-}{\mathbb Q}$.

If, in order to obtain more points, we remove the finiteness condition; then we have an analogue of Deligne's Theorem (see \cite[\S IX.11]{MacMoe}) which says that a coherent topos has enough points.

\begin{prop}\label{deligne} Every locally coherent Grothendieck category ${\mathcal G}$ has enough points in the sense that if $\alpha, \beta:G\rightrightarrows H$, $\alpha \neq \beta$, are distinct parallel morphisms of ${\mathcal G}$ then there is a geometric morphism $f:{\bf Ab} \rightarrow {\mathcal G}$ which distinguishes them in the sense that $f^\ast \alpha \neq f^\ast \beta$.
\end{prop}
\begin{proof} Since each of $\alpha, \beta$ is a direct limit of morphisms in ${\mathcal G}^{\rm fp}$ it is sufficient to assume that $\alpha, \beta$ are morphisms in ${\mathcal G}^{\rm fp}$ and hence can be regarded as pp-definable maps in ${\rm fun}({\mathcal D})$ where ${\mathcal D}= {\rm Ex}({\mathcal G}, {\bf Ab})$.  Since they are different, there is a model $D\in {\mathcal D}$ on which they are different, so ${\rm ev}_D:{\mathcal G}^{\rm fp} \rightarrow {\bf Ab}$ is such that ${\rm ev}_D(\alpha )\neq {\rm ev}_D(\beta)$.  Since ${\mathcal G}$ is the ${\rm Ind}$-completion of ${\mathcal G}^{\rm fp}$, this induces a $\varinjlim$-preserving functor $f^\ast:{\mathcal G} \rightarrow {\bf Ab}$ which, note, is also exact (every exact sequence in ${\mathcal G}$ is a directed colimit of exact sequences in ${\mathcal G}^{\rm fp}$).

We have to show that $f^\ast$ has a right adjoint and we have just noted that it preserves colimits so it remains to check the solution set condition of the adjoint functor theorem.  But that follows by El Bashir's theorem \cite[3.2]{ElBash}.
\end{proof}

Along the same lines, we can weaken the finiteness condition on $f^\ast$ by defining an $ \alpha $-coherent morphism $ f:{\mathcal G}\rightarrow {\mathcal H} $ between Grothendieck categories to be an adjoint pair  $\xymatrix{{\mathcal G} \ar@/_/[r]_{f_\ast}  & {\mathcal H} \ar@/_/[l]_{f^\ast}}$ with $ f^\ast $ exact such that $ f^\ast:{\mathcal H}^{<\alpha }\rightarrow {\mathcal G}^{<\alpha }$, where $ {\mathcal G}^{<\alpha } $ denotes the full subcategory of $ <\alpha$-presented objects of $ {\mathcal G}$.

\begin{lemma}\label{grothalpha}\marginpar{grothalpha} If  $\xymatrix{{\mathcal G} \ar@/_/[r]_{f_\ast}  & {\mathcal H} \ar@/_/[l]_{f^\ast}}$ is an adjoint pair of functors between Grothendieck categories with $ f^\ast $ left exact then there is $ \alpha  $ such that $ f^\ast {\mathcal H}^{<\alpha }\subseteq {\mathcal G}^{<\alpha } $ and so such that this morphism is $ \alpha $-coherent.
\end{lemma}
\begin{proof} (sketch)  There is $ \beta  $ such that $ {\mathcal H}^{<\beta } $ is generating and this category is skeletally small, so its image under $ f^\ast $ is contained in $ {\mathcal G}^{<\alpha '} $ for some $ \alpha '$. Then if $ \alpha  $ is the least regular cardinal greater than both $ \beta  $ and $ \alpha ' $ it will be the case that $ f^\ast{\mathcal H}^{<\alpha }\subseteq {\mathcal G}^{<\alpha }$.
\end{proof}

In the example mentioned earlier there were no non-zero coherent points and a perhaps more natural remedy is to note that there are enough ${\mathbb Q}$-points, meaning coherent morphisms from ${\rm Mod}\mbox{-}{\mathbb Q}$.
This fits with the (additive) ``functor-of-points" view of Section \ref{secfunpt}, in the terminology of which a geometric point of ${\mathcal G}$ is an ${\bf Ab}$-point of ${\mathcal G}^{\rm fp}$.  If it is a coherent point then one may see that it will induce a morphism of ``abelian spaces" $ {\mathcal O}_{\mathbb Z}\rightarrow {\mathcal O}_{\mathcal D} $ where ${\mathcal O}_{\mathbb Z}$ is $ {\rm Spec}({\mathbb Z})$ equipped with its usual structure sheaf and where ${\mathcal O}_{\mathcal D}$ is the natural presheaf of localisations of ${\rm fun}({\mathcal D})$ over ${\rm Zar}({\mathcal D})$ (as in \cite[\S 3]{PreRajShv}, but not sheafified).  Thus we come to the view that a coherent (${\bf Ab}$-)point of a locally coherent category $ {\mathcal G} $ is a morphism of ringed spaces from $ {\mathcal O}_{\mathbb Z} $ to a space which has the indecomposable injectives of ${\mathcal G}$ as points and which carries the Zariski topology (as defined by Gabriel \cite[\S V.4, Chpt.~VI]{Gab}, that is, in terms of injective modules rather than prime ideals).  As already remarked, if ${\mathcal G}$ is a $k$-category, where $k$ is a field, or just any commutative ring (for instance) then it would be more appropriate to use $ k$-points of $ {\mathcal G} $ and so to associate morphisms of ringed spaces from $ {\mathcal O}_k $ to the abelian space, ${\mathcal O}_{\mathcal G}$, associated to $ {\mathcal G}$.

\end{document}